\documentclass{amsart}

\usepackage{amsmath, amssymb, amsthm, amsfonts, marginnote}

\usepackage{color}
\input xy
\xyoption{all}

\usepackage{hyperref}
\usepackage{dsfont}

\swapnumbers
\numberwithin{equation}{section}

\theoremstyle{plain}
\newtheorem{theorem}[subsubsection]{Theorem}
\newtheorem{theorem-construction}[subsubsection]{Theorem-Construction}
\newtheorem{lemma}[subsubsection]{Lemma}
\newtheorem{prop}[subsubsection]{Proposition}
\newtheorem{cor}[subsubsection]{Corollary}
\newtheorem{conj}[subsubsection]{Conjecture}
\newtheorem*{claim}{Claim}

\theoremstyle{definition}

\newtheorem{remark}[subsubsection]{Remark}

\newtheorem{ex}[subsubsection]{Example}

\def\AA{\mathbb{A}}

\def\CC{\mathbb{C}}

\def\GG{\mathbb{G}}

\def\II{\mathbb{I}}

\def\KK{\mathbb{K}}
\def\LL{\mathbb{L}}

\def\NN{\mathbb{N}}

\def\QQ{\mathbb{Q}}
\def\RR{\mathbb{R}}

\def\TT{\mathbb{T}}

\def\ZZ{\mathbb{Z}}

\def\calB{\mathcal{B}}

\def\calE{\mathcal{E}}
\def\calF{\mathcal{F}}
\def\calG{\mathcal{G}}

\def\calK{\mathcal{K}}

\def\calN{\mathcal{N}}
\def\calO{\mathcal{O}}

\def\calS{\mathcal{S}}

\newcommand\cC{\mathcal{C}}
\newcommand\cD{\mathcal{D}}
\newcommand\cE{\mathcal{E}}
\newcommand\cF{\mathcal{F}}
\newcommand\cG{\mathcal{G}}

\newcommand\cI{\mathcal{I}}

\newcommand\cK{\mathcal{K}}
\newcommand\cL{\mathcal{L}}

\newcommand\cN{\mathcal{N}}
\newcommand\cO{\mathcal{O}}

\newcommand\cS{\mathcal{S}}

\newcommand\cU{\mathcal{U}}
\newcommand\cV{\mathcal{V}}

\newcommand\cZ{\mathcal{Z}}

\def\bR{\mathbf{R}}

\def\bT{\mathbf{T}}

\newcommand\frR{\mathfrak{R}}

\newcommand\frc{\mathfrak{c}}

\newcommand\frg{\mathfrak{g}}
\newcommand\frh{\mathfrak{h}}

\newcommand\frl{\mathfrak{l}}
\newcommand\fm{\mathfrak{m}}

\newcommand\frp{\mathfrak{p}}

\newcommand\frt{\mathfrak{t}}

\newcommand{\Bun}{\textup{Bun}}

\newcommand{\Der}{\textup{Der}}
\newcommand\ev{\textup{ev}}

\newcommand\Gal{\textup{Gal}}

\newcommand{\Gr}{\textup{Gr}}

\newcommand{\Hecke}{\textup{Hecke}}

\newcommand\IC{\textup{IC}}
\newcommand\id{\textup{id}}

\newcommand{\Ind}{\textup{Ind}}

\newcommand\Lie{\textup{Lie}}
\newcommand\Loc{\textup{Loc}}

\newcommand\Mod{\textup{Mod}}

\newcommand{\Perf}{\textup{Perf}}
\newcommand\Perv{\textup{Perv}}

\newcommand{\QCoh}{\textup{QCoh}}

\newcommand{\red}{\textup{red}}

\newcommand\Rep{\textup{Rep}}
\newcommand{\Res}{\textup{Res}}
\newcommand\res{\textup{res}}

\newcommand\Sat{\textup{Sat}}

\newcommand\Spec{\textup{Spec}\ }

\newcommand\St{\mathit{St}}

\newcommand\Sym{\textup{Sym}}

\newcommand{\univ}{\textup{univ}}

\newcommand\Aut{\textup{Aut}}
\newcommand\Hom{\textup{Hom}}
\newcommand\End{\textup{End}}

\newcommand\GL{\textup{GL}}

\newcommand\PGL{\textup{PGL}}

\newcommand\SL{\textup{SL}}

\newcommand{\Gm}{\GG_m}

\newcommand{\ad}{\textup{ad}}
\newcommand{\Ad}{\textup{Ad}}

\newcommand{\incl}{\hookrightarrow}
\newcommand{\isom}{\stackrel{\sim}{\to}}

\renewcommand{\a}{\alpha}

\newcommand{\g}{\gamma}

\newcommand\D{\Delta}

\renewcommand{\l}{\lambda}
\renewcommand{\L}{\Lambda}
\newcommand{\om}{\omega}

\renewcommand{\th}{\theta}
\newcommand{\Th}{\Theta}
\newcommand{\ph}{\varphi}
\newcommand{\z}{\zeta}

\newcommand{\twtimes}[1]{\stackrel{#1}{\times}}
\newcommand{\jiao}[1]{\langle{#1}\rangle}
\newcommand{\wt}[1]{\widetilde{#1}}
\newcommand{\wh}[1]{\widehat{#1}}
\newcommand\quash[1]{}
\newcommand\un{\underline}
\newcommand\ov{\overline}
\newcommand\bs{\backslash}

\newcommand\upH{\textup{H}}

\newcommand\tdN{\widetilde{\calN}^{\vee}}

\newcommand\dG{G^{\vee}}
\newcommand\dB{B^{\vee}}

\newcommand{\Max}{\textup{Max}}

\newcommand\hs{\heartsuit}

\newcommand\nat{\natural}
\newcommand\xr{\xrightarrow}

\newcommand{\Maps}{\textup{Maps}}

\newcommand{\Sh}{\mathit{Sh}}

\newcommand{\Coh}{\textup{Coh}}

\newcommand{\Coord}{\textup{Coord}}

\newcommand{\Hitch}{\textup{Hitch}}
\newcommand{\sing}{\textup{sing}}

\title[Spectral action in Betti Geometric Langlands]
{Spectral action in Betti Geometric Langlands}

\dedicatory{}
\author{David Nadler}
\thanks{}
\address{Department of Mathematics, UC Berkeley, Evans Hall, Berkeley, CA 94720}
\email{nadler@math.berkeley.edu}
\author{Zhiwei Yun}
\thanks{}
\address{Department of Mathematics, MIT,  77 Massachusetts Ave, Cambridge, MA 02139}
\email{zyun@mit.edu}
\date{}

\begin{document}

%%%%%%%%%%%%%%%%%%%%%%%%%%%%%%%%%%%%%%%%%%%%%%%%%%%%

\begin{abstract}
Let $X$ be a smooth projective curve, $G$ a reductive group, and $\Bun_G(X)$ the moduli of $G$-bundles on $X$.
For each point of $X$,  the Satake category acts by Hecke modifications on sheaves on $\Bun_G(X)$. We show that, for   sheaves with nilpotent singular support, the action is  locally constant with respect to the point of $X$.
This equips sheaves with nilpotent singular support with a module structure over  perfect complexes on the Betti moduli $\Loc_{\dG}(X)$ of dual group local systems. In particular, we establish the ``automorphic to Galois'' direction in the Betti Geometric Langlands correspondence  -- to each indecomposable automorphic sheaf, we attach a dual group local system  -- and define the Betti version of V.~Lafforgue's excursion operators.
\end{abstract}

%%%%%%%%%%%%%%%%%%%%%%%%%%%%%%%%%%%%%%%%%%%%%%%%%%%%

\maketitle

\tableofcontents

%%%%%%%%%%%%%%%%%%%%%%%%%%%%%%%%%%%%%%%%%%%%%%%%%%%%
%%%%%%%%%%%%%%%%%%%%%%%%%%%%%%%%%%%%%%%%%%%%%%%%%%%%
%%%%%%%%%%%%%%%%%%%%%%%%%%%%%%%%%%%%%%%%%%%%%%%%%%%%

\section{Introduction}

%%%%%%%%%%%%%%%%%%%%%%%%%%%%%%%%%%%%%%%%%%%%%%%%%%%%

%%%%%%%%%%%%%%%%%%%%%%%%%%%%%%%%%%%%%%%%%%%%%%%%%%%%
%%%%% red starts %%%%%%
%{\color{red}

\subsection{Motivation: the Betti Geometric Langlands}
In~\cite{BNbetti}, a Betti version of the geometric Langlands correspondence was formulated. Let us recall its statement.

Let $G$ be a  complex reductive group and $X$ a connected smooth projective complex curve. Let $\Bun_G(X)$ be the moduli stack of $G$-bundles on $X$. Fix a commutative coefficient ring $E$ that is noetherian and of finite global dimension.
Let $\Sh(\Bun_G(X), E)$ be the dg derived category of all complexes of $E$-modules on $\Bun_G(X)$.

Let
 $T^*\Bun_G(X)$ be  its  cotangent bundle (the underlying classical stack of the total space
of its cotangent complex), and let 
\begin{equation*}
\calN_G(X) \subset T^*\Bun_G(X)
\end{equation*}
 be the global nilpotent cone (the zero-fiber
of the Hitchin system).

%We will abuse terminology and use the term sheaves
%to refer to its objects.
To an object
$\cF\in \Sh(\Bun_G(X), E)$,
we can assign its singular support 
\begin{equation*}
\sing(\cF)\subset T^*\Bun_G(X)
\end{equation*}
 which is closed, conic, and coisotropic.

On the automorphic side,  we introduce the full dg subcategory  
\begin{equation*}
\Sh_{\calN_G(X)}(\Bun_G(X), E) \subset \Sh(\Bun_G(X), E)
\end{equation*}
 of complexes with singular support lying in $\calN_G(X)$. Since $\calN_G(X)$ is Lagrangian,
 for each object  of $\Sh_{\calN_G(X)}(\Bun_G(X), E)$, by \cite[Thm 8.5.5]{KS} there is a stratification of $\Bun_G(X)$ along which the object is locally constant (see the discussion and references of Section~\ref{sss:sing scheme} below). Thus objects in $\Sh_{\calN_G(X)}(\Bun_G(X), E)$ are weakly constructible, though their stalks do not necessarily satisfy the finite-dimensional
 cohomology requirement of constructibility.
 
\begin{remark}
It was conjectured by Laumon~\cite[Conj. 6.3.1]{L0} that cuspidal Hecke eigensheaves have nilpotent singular support. Thus the imposition of nilpotent singular support conjecturally keeps in play the most interesting automorphic sheaves. See the introduction of \cite{BNbetti} for detailed justification of imposing the nilpotent singular support condition.
\end{remark}

On the spectral side,  let $\dG$ be the Langlands dual group of $G$, viewed as a group scheme over $\ZZ$. Let $\dG_{E}$ be its base change to $E$. Let $\Loc_{\dG}(X)$ be the Betti derived stack over $\ZZ$ of topological $\dG$-local systems on $X$, and let $\Loc_{\dG}(X)_{E}$ be its base change to $E$.  
For a choice of base-point $x_0\in X$, we have the monodromy isomorphism   
\begin{equation*}
\xymatrix{
\Loc_{\dG}(X) \simeq \Hom(\pi_1(X, x_0), \dG)/\dG.
}
\end{equation*}
Let $\Ind\Coh_\calN(\Loc_{\dG}(X)_{E})$ denote the dg category of ind-coherent sheaves on $\Loc_{\dG}(X)_{E}$ with nilpotent
singular support  (in the sense developed by Arinkin and Gaitsgory \cite{AG}).

\begin{conj}[Rough form of Betti Geometric Langlands correspondence, see {\cite[Conjecture 1.5]{BNbetti}}]\label{main conj: intro betti} Let $E$ be a field of characteristic zero. Then there is an equivalence
\begin{equation}\label{eq: betti conj}
\xymatrix{
 \Ind\Coh_\calN ( \Loc_{\dG}(X)_{E}) \ar[r]^-\sim & \Sh_{\calN_{G}(X)}(\Bun_{G}(X), E)
}
\end{equation}
compatible with Hecke modifications and parabolic induction.
\end{conj}

\begin{remark}
There are natural generalizations of the conjecture for $G$-bundles with tame level structures
and corresponding $\dG$-local systems on open curves, see \cite[Conjecture 4.12]{BNbetti}.
\end{remark}

This paper contains three main results motivated by Conjecture \ref{main conj: intro betti}. 

First, in order for Conjecture \ref{main conj: intro betti} to  make sense, one needs to check that the Hecke functors on $\Sh(\Bun_G(X), E)$ in fact preserve the subcategory $\Sh_{\calN_{G}(X)}(\Bun_{G}(X), E)$.  Our first main result  (see Theorem \ref{thm: intro}) provides a strong version of this statement.

Second, on the spectral side, there is a natural tensor action of the tensor category
 of quasi-coherent complexes,
 or equivalently ind-coherent sheaves with trivial singular support
 \begin{equation*}
\xymatrix{
\QCoh(\Loc_{\dG}(X)_{E}) \simeq  \Ind\Coh_0 ( \Loc_{\dG}(X)_{E}) \curvearrowright \Ind\Coh_{\cN} ( \Loc_{\dG}(X)_{E}).
}
\end{equation*}
Therefore the Betti geometric Langlands conjecture predicts an action of $\QCoh(\Loc_{\dG}(X)_{E})$ on the automorphic category $\Sh_{\calN_{G}(X)}(\Bun_{G}(X), E)$. The construction of such an action is the second main result of this paper (see Theorem \ref{thm: intro action}). 
Such an action was used recently in~\cite{NY3pts} to construct the geometric Langlands correspondence in a special case, and we expect it to be a key tool in further developments in the Betti Geometric Langlands program.
 
Lastly, our third main result  (see Theorem \ref{cor: intro Langlands par}) applies the preceding to construct a Betti Langlands parameter to indecomposable automorphic sheaves. 
 
 We  formulate our main results immediately below in more detail.
%This paper provides a key foundational result in the Betti Geometric Langlands correspondence (see~\cite{BNbetti} for an outline of expectations). 

\subsection{Singular support under Hecke modifications}

Set $\calO=\CC[[t]]$ to be the power series ring, and $\calK=\CC((t))$ its fraction field.
Let $G(\calK)$ be the loop group, $G(\calO)$ its parahoric arc subgroup,
and $ \Gr_G = G(\calK)/G(\calO)$ the affine Grassmannian.

For each point $x\in X$, and the choice of a local coordinate at $x$, the Satake category $\Sat_G = \Sh_c(G(\cO)\backslash \Gr_G,  E)$ of $G(\cO)$-equivariant 
constructible complexes on $\Gr_G$ with compact support acts on $\Sh(\Bun_G(X), E)$ via Hecke modifications at $x$. 
For a fixed kernel $\cV \in \Sat_G$, this gives a Hecke functor
\begin{equation*}
\xymatrix{
H_{\cV, x}:\Sh(\Bun_G(X), E) \ar[r] &  \Sh(\Bun_G(X), E).
}
\end{equation*}
We will consider a version of the Satake category, denoted $\Sat^{0}_{G}$, whose objects carry equivariant structures for changes of the local parameters, see Section~\ref{sss:Sat0}. For $\cV\in \Sat^{0}_{G}$, one can define a family version of $H_{\cV,x}$ by allowing $x$ to vary
\begin{equation*}
\xymatrix{
H_{\cV}:\Sh(\Bun_G(X), E) \ar[r] &  \Sh(\Bun_G(X)\times X, E).
}
\end{equation*}
For example, we can take the kernel $\cV^\l \in \Sat^{0}_G$ given by the constant sheaf on a $G(\cO)$-orbit
$ \Gr^\l_G \subset \Gr_G $ in which case we write 
\begin{equation*}
\xymatrix{
H^\l = H_{\cV^\l}:\Sh(\Bun_G(X), E) \ar[r] &  \Sh(\Bun_G(X)\times X, E).
}
\end{equation*}
Conversely, these basic kernels generate all possibilities.

Here is our main technical result, proved below in Section~\ref{s: main thm}.

\begin{theorem}\label{thm: intro} 
For any kernel $\cV\in \Sat^{0}_{G}$, the Hecke functor $H_{\cV}$
preserves nilpotent singular support, 
and,  for sheaves with nilpotent singular support,
it does not introduce non-zero singular codirections along the curve: for $\cF\in \Sh(\Bun_G(X), E)$, we have
\begin{equation*}
\xymatrix{
\sing(\calF) \subset \calN_G(X) \implies 
\sing(H_{\cV}(\calF)) \subset \calN_G(X) \times X 
}
\end{equation*}
where $X\subset T^*X$ denotes the zero-section.
\end{theorem}

\begin{remark}
We also discuss in Section~\ref{s: variations} how Theorem~\ref{thm: intro} extends to $G$-bundles with level structures.
\end{remark}

%%%%%%%%%%%%%%%%%%%%%%%%%%%%%%%%%%%%%%%%%%%%%%%%%%%%

\subsection{Betti spectral action}

By the geometric Satake correspondence with general ring coefficients~\cite[(1.1)]{MV} (see \cite{Gi} for the case of complex coefficients), the convolution product on $\Sat_G$ preserves the perverse heart $\Sat^\heartsuit_G = \Perv_{c}(G(\cO)\backslash \Gr_G, E)$; the induced monoidal structure on $\Sat^\heartsuit_G$ 
extends to a tensor structure; and there is a natural
tensor equivalence
 \begin{equation*}
\xymatrix{
\Sat^\heartsuit_G \simeq \Rep(\dG_{E})
}
\end{equation*}
with the tensor category of representations of $\dG_{E}$ on finitely generated $E$-modules. 

%%%% begin blue %%%%
%{\color{blue} 

%Let $\Loc_{\dG}(X)$ be the Betti derived stack over $\ZZ$ of topological $\dG$-local systems on $X$.

Via the geometric Satake correspondence, $\Rep(\dG_{E})$ acts on $\Sh_{\calN_G(X)} (\Bun_G(X), E)$ by Hecke functors.  Theorem  \ref{thm: intro} implies the action is locally constant in the modification point $x\in X$ 
 (see Proposition~\ref{p: chiral}). 
In Section~\ref{ss: action} below, we deduce from this the following second main result.

%}
%%%% end blue %%%%

\begin{theorem}[Betti spectral action]\label{thm: intro action}
Let $E$ be a field of characteristic zero. Let $\Perf(\Loc_{\dG}(X)_{E})$ be the tensor dg category of perfect complexes on $\Loc_{\dG}(X)_{E}$. Then there is an $E$-linear tensor action
\begin{equation*}
\xymatrix{
\Perf(\Loc_{\dG}(X)_{E}) \curvearrowright \Sh_{\calN_G(X)} (\Bun_G(X), E) }
\end{equation*}
%\begin{equation}
%\xymatrix{
%\int_X \Rep(\dG,E)\curvearrowright \Sh_{\calN_G(X)} (\Bun_G(X), E) }
%\end{equation}
such that for any point $x\in X$, its restriction via pullback along the natural evaluation
 \begin{equation*}
\xymatrix{
\Rep(\dG_{E}) \ar[r]^-{\ev_{x}^{*}} & \Perf(\Loc_{\dG}(X)_{E})}
\end{equation*}
is isomorphic, under the Geometric Satake correspondence, to the Hecke action
of $\Sat^\heartsuit_G$ at the point $x\in X$.
\end{theorem}

%\begin{remark} Note that the above theorem holds for arbitrary coefficient ring $E$, and in particular for fields $E$ of arbitrary characteristic. Of course, the tensor action of $\Perf(\Loc_{\dG}(X))$ can also be upgraded to an $E$-linear tensor action of $\Perf(\Loc_{\dG}(X)_{E})$, where we write $\Loc_{\dG}(X)_{E} $ for the base change of $\Loc_{\dG}(X)$ from $\ZZ$ to $E$. 
%\end{remark}

\begin{remark}
In the setting of $\cD$-modules with no prescribed singular support, the construction of an analogous action of quasi-coherent sheaves on the stack of de Rham connections is a deep ``vanishing theorem" whose proof is sketched by Gaitsgory in ~\cite{Gvanish}.
\end{remark}

%%%% begin blue %%%%
%{\color{blue} 
\begin{remark} \label{rem: chiral homol}
In Section~\ref{ss: action} below, we deduce Theorem~\ref{thm: intro action} ``by hand", but
 one can  also appeal to the general machinery of topological chiral homology (\cite[\S5.5.4]{HA}) as we sketch here.
   
   To the tensor dg category $\Perf(B\dG_{E})$ of perfect complexes and curve $X$, one  can assign the topological chiral homology $\int_X \Perf(B\dG_{E})$.
  It is again a tensor dg category and comes equipped  with a tensor functor from $\Perf(B\dG_{E})$, for each $x\in X$. In fact, it is universal for having such functors along with equivalences between them along paths, together with higher coherences along higher simplices.
    When $E$ is a field of characteristic zero, the results of~\cite{BFN} provide a tensor equivalence
\begin{equation*}
\int_X \Perf(B\dG_{E})\simeq \Perf(\Loc_{\dG}(X)_{E})
\end{equation*}
compatible with the tensor functors from $\Perf(B\dG_{E})$, for each $x\in X$.

Recall that Theorem  \ref{thm: intro} implies the   $\Perf(B\dG_{E})$-action on $\Sh_{\calN_G(X)} (\Bun_G(X), E)$ is locally constant in the modification point $x\in X$.
It follows from the universal property that the action descends to a $\int_X \Perf(B\dG_{E})$-action, 
 and thus a
 $ \Perf(\Loc_{\dG}(X)_{E})$-action,
 as asserted in Theorem~\ref{thm: intro action}.
\end{remark}

%}
%%%% end blue %%%%

\subsection{Betti Langlands parameters}When $E$ is a field of characteristic zero, 
one can use the action of Theorem~\ref{thm: intro action} to associate a Betti Langlands parameter to an indecomposable constructible automorphic complex with finite type support.  
Here, as usual, indecomposable means the complex cannot be expressed as a direct sum of non-trivial summands.

%
%{\color{red} By a  
%Betti Langlands parameter, we mean the following invariant of Hecke modifications.}
When $E$ is  an algebraically closed field of characteristic zero,  the action of $\Perf(\Loc_{\dG}(X)_{E})$ on $\Sh_{\calN_G(X)} (\Bun_G(X), E)$ of 
Theorem \ref{thm: intro action} implies that the dg algebra $\cO(\Loc_{\dG}(X)_{E})$ acts on each object $\cF\in \Sh_{\calN_G(X)} (\Bun_G(X), E)$. Moreover, when $\cF$ satisfies certain finiteness properties, for example when $\cF$ is an irreducible perverse sheaf, this action allows us to assign a maximal ideal of $\cO(\Loc_{\dG}(X)_{E})$ to $\cF$, which determines a $\dG(E)$-local system $\rho_{\cF}$ on $X$ up to semisimplification. 
When $E$ is of arbitrary characteristic, we do not establish the categorical action analogous to that in Theorem~\ref{thm: intro action}.  Nevertheless the assignment $\cF\mapsto \rho_{\cF}$ can be constructed, as Theorem \ref{cor: intro Langlands par}  below states.

Consider the  full subcategory 
 \begin{equation*}
 \Sh_{\calN_G(X),!}(\Bun_{G}(X), E) \subset
 \Sh_{\calN_G(X)} (\Bun_G(X), E)
 \end{equation*}
   of objects of the form $j_{!}\cF_{\cU}$, where $j:\cU\incl \Bun_{G}(X)$ is an open embedding of a finite type substack, and $\cF_{\cU}$ is a constructible complex on $\cU$. The restriction to such complexes is used to guarantee  finiteness of their endomorphisms (see the proof of Theorem-Construction~\ref{cor: Langlands par}). The following 
   third main result will be proved in Section~\ref{s:Betti par}.

\begin{theorem}[Betti Langlands parameter]\label{cor: intro Langlands par} 
Let $E$ be an algebraically closed field. To any indecomposable object $\cF\in  \Sh_{\calN_{G}(X),!}(\Bun(X), E)$ one can canonically attach a semisimple $\dG(E)$-local system  $\rho_{\cF}$ over $X$. Moreover, if $\cF$ is a Hecke eigensheaf with eigenvalue  $\rho\in \Loc_{\dG}(X)(E)$, then $\rho_{\cF}$ is isomorphic to the semisimplification of $\rho$.
\end{theorem}
In the arithmetic Langlands program, the Langlands parameter of an automorphic representation of $G$ over a global field $F$ is a continuous homomorphism $\rho: \Gal(F^{s}/F)\to \dG(\overline{\QQ}_{\ell})$. In the geometric setting, the fundamental group of $X$ plays the role of $\Gal(F^{s}/F)$, and $\dG$-valued Galois representations $\rho$ are replaced by $\dG$-local systems on $X$. Therefore it makes sense to call the $\dG(E)$-local system $\rho_{\cF}$ constructed in Theorem \ref{cor: intro Langlands par} the (semisimple) {\em Betti Langlands parameter} of $\cF$.

\begin{remark} There is also a version of Theorem~\ref{thm: intro action} and Theorem~\ref{cor: intro Langlands par}  in the presence of level structures. See Section~\ref{s: variations}.
\end{remark}

%
%
%\begin{cor}\label{cor: intro}
%The Hecke functors $H_{\cV, x}$ at points $x\in X$ with kernels $\cV\in \Sat^\heartsuit_G \simeq \Rep(\dG)$ descend 
%to a tensor action
% \begin{equation}
%\xymatrix{
%\Perf(\Loc_{\dG}(X)) \otimes \Sh_{\calN_G(X)} (\Bun_G(X)) \ar[r] &  \Sh_{\calN_G(X)} (\Bun_G(X))
%}
%\end{equation}
%
%\end{cor}
%

\begin{remark} Theorem~\ref{cor: intro Langlands par} can be viewed as a categorical analogue of the main result of V.~Lafforgue from \cite{Laf}. Roughly speaking, when $X$ is defined over a finite field $k$, V.~Lafforgue constructs an action of the coordinate ring of the $\dG$-character variety of the absolute Galois group of $F=k(X)$ on the space of cuspidal automorphic forms on $G(F)\bs G(\AA_{F})$. This allows him to attach a semisimple Galois representation to each irreducible cuspidal automorphic representation of $G(\AA_{F})$. In our situation, the ring of regular functions on $\Loc_{\dG}(X)$ acts on each object of $\Sh_{\calN_G(X)}(\Bun_{G}(X), E)$, and we can attach a semisimple $\dG$-local system to each indecomposable automorphic complex.  V.~Lafforgue's construction relies crucially on the partial Frobenius structure of the moduli of Shtukas. In a vague sense, the nilpotent singular support condition we impose on complexes on $\Bun_{G}(X)$ is playing a similar role as the partial Frobenius:  in both cases they are ensuring certain local constancy of Hecke modifications. 
\end{remark}

\begin{remark}
For number theorists: the construction of the Betti Langlands parameters in Theorem~\ref{cor: intro Langlands par} uses an analogue of the $\bR\to \bT$ map (a ring homomorphism from a deformation ring of Galois representations to a Hecke ring). When $E$ is an algebraically closed field, we construct in  Corollary~\ref{cor: R to T} a diagram of rings
\begin{equation}\label{intro RT}
\xymatrix{  R=\upH^{0}(\Loc_{\dG}(X)_{E}) & R^{\univ}\ar[r]^-{\sigma}\ar[l]_-{\om} & \cZ(\Sh_{\calN_{G}(X)} (\Bun_{G}(X), E))
}
\end{equation} 
where the right end is the center (i.e., endomorphisms of the identity functor) of $\Sh_{\calN_{G}(X)} (\Bun_{G}(X), E)$, and the ring $R^{\univ}$ in the middle is defined by a universal property. Moreover, the map $\Spec R\to \Spec R^{\univ}$ induced by
$\om$ is a bijection on closed points. The construction of the diagram uses analogues of V.~Lafforgue's excursion operators in the Betti setting. For details, see Section~\ref{s:Betti ex}.

On the other hand, for topologists: for general $E$, the symmetric monoidal structure on $\Rep(\dG_{E})$ equips  the topological chiral homology $\int_{X}\Rep(\dG_{E})$ with a symmetric monoidal structure as well. Theorem ~\ref{thm: intro} and the formalism of topological chiral homology give an action diagram
\begin{equation*}
\xymatrix{    \Perf(\Loc_{\dG}(X)_{E})  \curvearrowleft  \int_{X}\Rep(\dG_{E})   \curvearrowright   \Sh_{\calN_{G}(X)} (\Bun_{G}(X), E). }
\end{equation*}
The monoidal unit $\mathds 1 \in \int_{X}\Rep(\dG_{E})$ acts on the module categories by the identity functor, so 
we obtain a diagram of derived rings 
\begin{equation}\label{intro RT2}
\xymatrix{  \cO(\Loc_{\dG}(X)_{E}) & \End(\mathds 1) \ar[r]\ar[l] & \tilde \cZ(\Sh_{\calN_{G}(X)} (\Bun_{G}(X), E)) }
\end{equation}
where the right end denotes  the derived center of  $\Sh_{\calN_{G}(X)} (\Bun_{G}(X), E))$. 
A natural problem is to study the relationship between $R^{\univ}$ and  $\End(\mathds 1)$,
and to compare~\eqref{intro RT} and \eqref{intro RT2}.
\end{remark}

%%%%%%%%%%%%%%%%%%%%%%%%%%%%%%%%%%%%%%%%%%%%%%%%%%%%

\subsection{Acknowledgements} 
We thank David Ben-Zvi for many inspiring discussions about a Betti form of the  Geometric Langlands
correspondence. We thank an anonymous referee for generous comments.

DN is grateful for the support of NSF grant DMS-1502178.
ZY is  grateful for the support of NSF grant DMS-1302071/DMS-1736600 and the Packard Foundation.

%%%%%%%%%%%%%%%%%%%%%%%%%%%%%%%%%%%%%%%%%%%%%%%%%%%%
%%%%%%%%%%%%%%%%%%%%%%%%%%%%%%%%%%%%%%%%%%%%%%%%%%%%
%%%%%%%%%%%%%%%%%%%%%%%%%%%%%%%%%%%%%%%%%%%%%%%%%%%%

\section{Notation}

\subsection{Automorphic side}\label{ss:auto notation}
All automorphic moduli stacks in this paper will be defined over $\CC$.

Let $G$ be a reductive group, $B\subset G$ a Borel subgroup, $N \subset B$ its unipotent radical,
and $T = B/N$ the universal Cartan. 
Let $\calB\simeq G/B$ be the flag variety of $G$.

Let $(\Lambda_T, R^{\vee}_+, \Lambda_T^\vee, R_+)$ be the associated based root datum, where 
$\Lambda_T =\Hom(\GG_m, T)$ is the coweight lattice,
$R^{\vee}_+\subset \Lambda_T$ the positive coroots, $\Lambda_T^\vee = \Hom(T, \GG_m)$ the weight lattice, and 
$R_+\subset \Lambda_T^\vee$ the positive roots. Let $\L^{+}_{T}$ (resp. $\L^{\vee,+}_{T}$) be the set of dominant coweights (resp. dominant weights).
%Let $\Wf$ denote the Weyl group of $G$, and $W^\aff \simeq \Wf\ltimes \Lambda_T$ its affine Weyl group. 
%Let $\rho\in \L^{\vee}_{T}$ (resp. $\rho^{\vee}\in \L_{T}$) be half of the sum of elements in $R_{+}$ (resp. $R^{\vee}_{+}$).

Fix a commutative coefficient ring $E$ that is noetherian and of finite global dimension. We will work in the setting of $E$-linear dg categories. Most of our categories will comprise complexes of sheaves of $E$-modules over the classical topology of (automorphic) stacks over $\CC$, e.g., $\Bun_{G}(X)$. All sheaf-theoretic functors are understood to be derived functors.

%For our geometric arguments, including the main results of Section~\ref{s: micro}, this could mean a traditional commutative ring, or a version of ``derived commutative ring"
%(as discussed in ~\cite[Sect. 2.3]{BFN}), for example, a connective commutative dga, simplicial commutative ring, or connective $E_\infty$-ring spectrum. For applications to spectral actions beginning in Section~\ref{ss: action},
%we invoke the Geometric Satake equivalence, in particular its abelian form, and restrict to  a
%traditional commutative ring. 

%One exception is the Satake category $\Sat_G = \Sh_c(G(\cO)\backslash \Gr_G, \ZZ)$, for which we consider complexes of sheaves of abelian groups.
%All automorphic sheaves and dg categories will be defined with respect to 
%For example, we will talk about complexes of sheaves of $E$-modules on $k$-stacks, such as $\Bun_G(X)$. 
%This can be understood to mean:
%\begin{itemize}
%\item When $k=\CC$, complexes of sheaves of $E$-vector spaces over the classical topology;
%
%\marginpar{When $k=\CC$. It looks like that we may work with any coefficient ring $E$, provided the basic functoriality of singular support hold for sheaves of $E$-modules. This should give Betti spectral decomposition in the modular setting.}
%
%
%\item When $k$ is arbitrary (but algebraically closed, characteristic zero), we take $E$ to be either a finite ring an algebraic extension of $\QQ_{\ell}$ for some prime $\ell$, and consider complexes of sheaves of $E$-vector spaces over the lisse-\'etale topology, in the sense of  \cite{LZ1,LZ2}.
%
%
%\end{itemize}
%

\subsection{Spectral side}\label{ss:spec notation}
All stacks on the Langlands dual side are defined over $\ZZ$, and we often take their base change to $E$ by adding a subscript $E$. 

Form the dual based root datum
$(\Lambda_T^\vee, R_+, \Lambda_T, R^{\vee}_+)$, and
construct the Langlands dual group $\dG$ (over $\ZZ$), with Borel subgroup $B^\vee\subset \dG$,
unipotent radical $N^\vee\subset B^\vee$,
and dual universal Cartan $T^\vee = B^\vee/N^\vee$. 
Let $\calB^\vee\simeq \dG/B^\vee$ be the flag variety of $\dG$.

 Let $\calN^\vee$ be the nilpotent cone in the Lie algebra $\frg^{\vee}$.
We identify $\calN^\vee $ with the unipotent elements in $\dG$ via the exponential map.
% $$\xymatrix{
% \exp:\calN^\vee\ar[r]^-\sim &  \cU^\vee
% }
% $$
 
Let $\mu:  \tilde \calN^\vee \to \calN^\vee$ be the Springer resolution. Recall that $\tilde \calN^\vee\subset \dG \times \calB^\vee$ classifies pairs $(g,\dB_{1})$
such that the class $g$ lies in the unipotent radical of $\dB_{1}$.  Note the isomorphism of adjoint quotients $N^\vee/B^\vee \simeq \tilde \calN^\vee/\dG$.

%All spectral moduli  in this paper,   such as $\Loc_{\dG}(X)$, will be defined over $\ZZ$, however, we will base change them to the coefficient field $E$. 

%%%%%%%%%%%%%%%%%%%%%%%%%%%%%%%%%%%%%%%%%%%%%%%%%%%%

\section{Local constructions}\label{s: local}

%%%%%%%%%%%%%%%%%%%%%%%%%%%%%%%%%%%%%%%%%%%%%%%%%%%%

\subsection{Automorphisms of disk}\label{ss: disk auts}

Set $\calO = \CC[[t]]$ to be the power series ring, with maximal ideal $\fm_\calO = t \CC[[t]]$, and fraction field $\calK = \CC((t))$.

Let $D = \Spec \calO$ be  the formal disk, and $D^\times  = \Spec \calK \subset D$ the formal punctured disk.

Let $\Aut^0(\calO) = \Spec \CC[c_1, c_1^{-1}, c_2, c_3,\ldots]$ be the group-scheme of automorphisms of $\calO$
that preserve the maximal ideal $\fm_\calO$. A point of $\Aut^0(\calO)$ with coordinate $(c_{1},c_{2},\cdots)$ corresponds to the automorphism $f(t)\mapsto f(c_{1}t+c_{2}t^{2}+\cdots)$ of $\cO$.
Let $\Aut(\calO) = \cup_{n \in \NN} \Spec \CC[c_0, c_1, c_1^{-1}, c_2, c_3,\ldots]/(c_0^n)$ be the group ind-scheme of automorphisms of $\calO$. Similarly, a point of $\Aut(\calO)$ with coordinate $(c_{0},c_{1},c_{2},\cdots)$, where $c_{0}$ is nilpotent, corresponds to the automorphism $f(t)\mapsto f(c_{0}+c_{1}t+c_{2}t^{2}+\cdots)$ of $\cO$.

We have
\begin{equation*}
\xymatrix{\Lie\Aut(\calO) = \Der(\calO) = \calO\partial_t
}
\end{equation*} 
and $\Aut(\calO)_\red = \Aut^0(\calO)$, so that $(\Der(\calO), \Aut^0(\calO))$ forms a Harish Chandra pair for $\Aut(\calO)$. 
  
Note as well 
\begin{equation*}
\xymatrix{\Der(\calO)^* \simeq   (\calK/\calO) \otimes_{\calO} \Omega_\calO  \simeq (\calK/\calO)dt
}
\end{equation*}
where $\Omega_\calO = \calO dt$. The $\CC$-linear pairing between $\Der(\cO)$ and $(\cK/\cO)dt$ is given by the residue pairing
  \begin{equation*}
  \xymatrix{
  \langle a(t)\partial_t, b(t)dt\rangle = \Res_{t=0} (a(t)b(t)).
  }
  \end{equation*}

%%%%%%%%%%%%%%%%%%%%%%%%%%%%%%%%%%%%%%%%%%%%%%%%%%%%

\subsection{Affine Grassmannian}

Introduce the Laurent series loop group $G(\calK)= \Maps(D^\times, G)$, its parahoric arc subgroup $G(\calO) = \Maps(D, G)$,
and affine Grassmannian $ \Gr_G = G(\calK)/G(\calO)$. Recall that $\Gr_G$ classifies the data of a $G$-bundle $\calE$ over $D$
with a trivialization (equivalently, section) of the restriction $\calE|_{D^\times}$. (We will  be exclusively interested in constructible sheaves on $\Gr_G$, and hence ignore the non-reduced structure arising when $G$ is not semisimple.)

Note that $\Aut(\cO)$ naturally acts on $G(\cK)$ preserving the subgroup $G(\cO)$,
and hence also acts on the affine Grassmannian $\Gr_G$. 

\subsubsection{Stratification by $G(\cO)$-orbits}
The inclusion $ \Lambda_T = \Hom(\GG_m, T) \hookrightarrow G(\calK)$, $\lambda\mapsto t^\l$ induces a bijection of sets 
\begin{equation*}
\xymatrix{
\Lambda_T^+ \simeq \Lambda_T/W \ar[r]^-\sim & G(\cO)\backslash G(\calK)/G(\cO)
&
\lambda \ar@{|->}[r] & G(\calO)\cdot t^\lambda \cdot G(\cO)
}
\end{equation*}

For  each $\lambda\in \Lambda_T^+$, set
\begin{equation*}
\xymatrix{
G(\calK)^\l = G(\calO)\cdot t^\lambda \cdot G(\calO) \subset G(\calK)
&
\Gr_G^\lambda = G(\calO)\cdot t^\lambda \subset \Gr_G
}
\end{equation*}
and also  let 
$\ov{G(\cK)}^\lambda\subset G(\cK)$, $\ov\Gr_G^\lambda\subset \Gr_G$  denote their respective closures.

Recall that $G(\cK)^\mu \subset \ov{G(\cK)}^\lambda$ if and only if $\Gr_G^\mu \subset \ov\Gr_G^\lambda$ if and only if $\mu \leq \lambda$ (in the sense that
$\lambda-\mu$ is a $\ZZ_{\geq 0}$-combination of simple coroots).

It is well-known that $\ov {G(\cK)}^\lambda\subset G(\cK)$ is a scheme (not locally of finite type), hence $G(\cK)$ is an increasing union of schemes, and
$\ov\Gr_G^\lambda\subset \Gr_G$ is a (typically singular) projective variety,  hence 
$\Gr_G$ is an increasing union of projective varieties.

The subgroup $\Aut^{0}(\cO)$ preserves $\Gr_{G}^{\l}$ and $\ov\Gr_{G}^{\l}$ for each $\l\in\L_{T}^{+}$.

%%%% 

\subsubsection{Orbit closure resolutions}\label{s:Gr res}

It will  be technically useful to have on hand a resolution of  $\ov\Gr_G^\lambda\subset \Gr_G$.

Let us choose a $G(\cO)\rtimes \Aut^0(\cO)$-equivariant 
resolution of singularities 
\begin{equation*}
\xymatrix{
\nu^{\l}: \wt \Gr_G^\l \ar[r] & \ov\Gr_G^\l 
}
\end{equation*}
To achieve this via general theory, observe that the $G(\cO)\rtimes \Aut^0(\cO)$-action on $\ov \Gr_G^\l$ factors through a finite-type group, so such a resolution exists by \cite[Prop.~3.9.1]{Ko}.
%so we are asking for a resolution of a projective variety, equivariant for an affine group.
Note this pulls back to  a $(G(\cO) \times G(\cO))\rtimes \Aut^0(\cO)$-equivariant  resolution
\begin{equation*}
\xymatrix{
\wt G(\cK)^\l \ar[r] & \ov{G(\cK)}^\l 
}
\end{equation*}
%and an $\Aut^0(\cO)$-equivariant resolution
%\begin{equation*}
%\xymatrix{
% G(\cO)\backslash \wt \Gr_G^\l \ar[r] &  G(\cO)\backslash \ov \Gr_G
%}
%\end{equation*}

%Note also  since $\nu^{\l}: \wt \Gr_G^\l \to \ov\Gr_G^\l$  is $G(\cO)$-equivariant, it is a fibration above each $\Gr^\mu_G \subset\ov \Gr^\l_G$, for $\mu\leq \l$. 

\begin{remark}
It suffices to choose   the resolutions $\wt \Gr_G^{\l_i} \to \ov \Gr_G^{\l_i}$
for a generating collection of coweights $\l_i \in \Lambda_T^+$,
and then in general take a  convolution space
\begin{equation*}
\xymatrix{
\wt \Gr_G^\l = \wt G(\cK)^{\l_1} \overset{G(\cO)}{\times} \wt G(\cK)^{\l_2} \overset{G(\cO)}{\times} \cdots 
\overset{G(\cO)}{\times} \wt \Gr_G^{\l_\ell}
&
\l  = \sum_{i=1}^\ell \l_i.
}
\end{equation*}
In particular, when $G$ has  a generating collection of minuscule coweights $m_i \in \Lambda_T^+$, so with smooth projective $G(\cO)$-orbits $\Gr_G^{m_i} = \ov\Gr_G^{m_i}$, we can simply take the convolution space
\begin{equation*}
\xymatrix{
\wt \Gr_G^\l =  G(\cK)^{m_1} \overset{G(\cO)}{\times}  G(\cK)^{m_2} \overset{G(\cO)}{\times} \cdots 
\overset{G(\cO)}{\times}  \Gr_G^{m_\ell}
&
\l  = \sum_{i=1}^\ell m_i.
}
\end{equation*}
\end{remark}

%%%%%%%%%%%

\subsection{Moment map}  The action of $\Aut(\cO)$ on $G(\cK)$ and $\Gr_{G}$ induces infinitesimal $\Der(\cO)$-actions, or equivalently moment maps
\begin{equation*}
\xymatrix{T^{*}G(\cK)\ar[r] & \Der(\cO)^{*},}
\end{equation*}
\begin{equation*}
\xymatrix{T^{*}\Gr_{G}\ar[r]& \Der(\cO)^{*}.}
\end{equation*}
We record here formulae for these moment maps.

\subsubsection{Loop group}
The infinitesimal $\Der(\calO)$-action on $G(\cK)$ is given
at a point $g(t)\in G(\cK)$ by the formula 
 \begin{equation}\label{eq: inf formula}
 \xymatrix{
\calO\partial_t =  \Der(\calO) \ar[r]  & T_{g(t)} G(\cK) \simeq  \frg(\cK)
& 
f(t)\partial_{t}\ar@{|->}[r] &  f(t)g'(t)g(t)^{-1}. 
}
\end{equation}
Here the symbol $g'(t)g(t)^{-1}$  denotes an element of $\frg(\cK)$ defined in the following way. For each $\cK$-linear $G(\cK)$-representation $(V,\rho)$, with associated $\frg(\cK)$-module $(V, d\rho)$, we obtain  an endomorphism $\rho(g(t))' \in \End_\cK(V)$ by differentiating the matrix coefficients of $\rho(g(t))$. Set $Y_{V}=\rho(g(t))'\rho(g(t))^{-1}\in\End_{\cK}(V)$. For two such representations $(V_{1},\rho_{1})$, $(V_{2}, \rho_{2})$, it is easy to check that $Y_{V_{1}\otimes V_{2}}=Y_{V_{1}}\otimes\id_{V_{2}}+\id_{V_{1}}\otimes Y_{V_{2}}$. Therefore by Tannakian formalism, there is a well-defined element $Y\in \frg(\cK)$ such that $d\rho(Y)=Y_{V}$, for all such representations $(V,\rho)$. We use the notation  $g(t)'g(t)^{-1}$ to denote this element $Y$.

To reformulate \eqref{eq: inf formula} as a moment map,  
%fix a $G$-invariant   non-degenerate pairing $ \langle \;, \,\rangle_{\frg}:\frg\otimes \frg\to k$ on the Lie algebra $\frg$ over $k$,
 consider the   $G(\calK)$-equivariant isomorphism
\begin{equation}\label{eq:gK dual}
\xymatrix{\frg(\cK)^* \simeq \frg^{*}(\cK) \otimes_\cO \Omega_\cO = \frg^{*}(\cK)dt
}
\end{equation}
given by the pairing
  \begin{equation*}
  \xymatrix{
  (v(t), w(t)dt) = \Res_{t=0} \langle v(t), w(t)dt\rangle, &   v(t)\in \frg(\cK), w(t)\in \frg^{*}(\cK).
  }
  \end{equation*}
%$\frg[[t]]^\perp \simeq \frg[[t]]$.
Then the moment map for the induced $\Aut(\calO)$-action on  the cotangent bundle $T^*G(\cK)$ is given at a point $g(t)\in G(\cK)$ by the  formula 
 \begin{equation*}
 \xymatrix{
  \frg^{*}(\cK)dt  \simeq \frg(\cK)^*  \simeq T^*_{g(t)} G(\cK) \ar[r] & (\Lie \Aut(\cO))^* \simeq \Der(\calO)^* 
}
\end{equation*}
\begin{equation}\label{eq: moment map}
\xymatrix{
\phi(t)dt \ar@{|->}[r] &  ( f(t)\partial_t \ar@{|->}[r] & \Res_{t=0}\langle f(t)g'(t) g(t)^{-1}, \phi(t)dt\rangle )
}
\end{equation}
Under the further isomorphism $\Der(\cO)^*\simeq (\cK/\cO)dt$ given by the residue pairing, this takes the form
 \begin{equation*}
\xymatrix{
\phi(t)dt \ar@{|->}[r] &   \langle g'(t) g(t)^{-1}, \phi(t)dt\rangle\in (\cK/\cO) dt  
}
\end{equation*}

\subsubsection{Affine Grassmannian}
We describe the cotangent bundle $T^*\Gr_{G}$ and give a similar formula for the moment map of the induced $\Aut(\cO)$-action. 

Let $g(t)G(\cO)\in \Gr_{G}$. Since the $G(\cK)$-action on $\Gr_{G}$ is transitive, the tangent space $T_{g(t)G(\cO)}\Gr_{G}$ is naturally isomorphic to the quotient
\begin{equation*}
\xymatrix{T_{g(t) G(\cO)} \Gr_G \simeq  \frg(\cK)/ \Ad_{g(t)}\frg(\cO)
}
\end{equation*}
Taking  duals and using $\frg(\cO)^\perp \simeq \frg^*(\cO) dt$ under the adjoint-equivariant identification \eqref{eq:gK dual}, we get a canonical isomorphism
\begin{equation}\label{eq:cot Gr}
\xymatrix{T^{*}_{g(t)G(\cO)}\Gr_{G}\simeq \Ad_{g(t)}(\frg^{*}(\cO)dt)
}
\end{equation}
For the induced $\Aut(\cO)$-action on $T^*\Gr_{G}$, 
the moment map 
\begin{equation*}
\xymatrix{
%\mu_{g(t)G(\cO)}: 
T^{*}_{g(t)G(\cO)}\Gr_{G}\simeq \Ad_{g(t)} (\frg^{*}(\cO)dt)\ar[r] &\Der(\cO)^{*} 
}
\end{equation*}
is given by the restriction of \eqref{eq: moment map}.

\begin{remark}
Recall that the $\Aut^0(\cO)$-action on $\Gr_G$  preserves the $G(\cO)$-orbits $\Gr_G^\l \subset \Gr_G$, for $\lambda\in \Lambda_T^+$. But from formula ~\eqref{eq: moment map}, one can see   the $\Aut(\cO)$-action 
does not preserve each  $\Gr_G^\l \subset \Gr_G$ as 
the action of $\partial_t \in \Der(\cO)$ is not tangent to $\Gr_G^\l \subset \Gr_G$.
\end{remark}

%%%%%%%%%%%%%%%%%%%%%%%%%%%%%%%%%%%%%%%%%%%%%%%%%%%%

\subsection{Satake category}

Let $\Sat_G = \Sh_c(G(\cO)\backslash \Gr_G, E)$ be the dg category of $ G(\cO)$-equivariant constructible complexes of $E$-modules on $\Gr_G$
with compact support.
 Convolution equips $\Sat_G$ with a monoidal structure with monoidal unit the skyscraper sheaf at the base-point
 $\Gr^0_G \subset \Gr_G$.

Recall by the geometric Satake correspondence with ring coefficients~\cite[(1.1)]{MV}, the convolution product on $\Sat_G$ preserves the perverse heart $\Sat^\heartsuit_G = \Perv_{c}(G(\cO)\backslash \Gr_G, E)$; the induced monoidal structure on $\Sat^\heartsuit_G$ 
extends to a tensor structure; and there is a natural
tensor equivalence
 \begin{equation*}
\xymatrix{
\Sat^\heartsuit_G \simeq \Rep(\dG_{E})
}
\end{equation*}
with the tensor category of representations of the dual group $\dG_{E}$ on finitely generated $E$-modules. 
Note the $G(\cO)$-equivariance of any object of 
$\Sat^\heartsuit_G $ is a property not an additional structure, and in particular, equivalent to its constructibility along the $G(\cO)$-orbits.

\subsubsection{Equivariance for disk automorphisms}\label{sss:Sat0}

Let $\Sat^0_G = \Sh_c((\Aut^0(\cO) \ltimes G(\cO))\backslash \Gr_G, E)$ be the dg category of $ \Aut^0(\cO) \ltimes G(\cO)$-equivariant constructible complexes on $\Gr_G$
with compact support.
We can equivalently view $\Sat^0_G$ as the dg category of $\Aut^0(\cO)$-invariants in $\Sat_G$.
From this perspective, the $\Aut^0(\cO)$-action naturally factors through the evaluation $\Aut^0(\cO)\to \GG_m$.

Convolution equips $\Sat^0_G$ with a monoidal structure with monoidal unit the skyscraper sheaf at the base-point
 $\Gr^0_G \subset \Gr_G$ with its natural $\Aut^0(\cO)$-equivariance.
The  forgetful functor $\Sat^0_G \to \Sat_G$ is monoidal and restricts to an equivalence of perverse hearts $\Sat^{0, \heartsuit}_G \stackrel{\sim}{\to} \Sat_G^\heartsuit$ (\cite[Prop. A.1]{MV}).

\begin{ex}
For each $\l\in \Lambda^+_T$, the constant sheaf  $\un E_{\Gr^\l_G}$ on the $G(\cO)$-orbit $\Gr^\l_G \subset \Gr_G$ (extended by zero to $\Gr_{G}$), 
the constant sheaf $ \un E_{\ov \Gr^\l_G}$ 
and intersection complex $\IC^\l$ (with $ E$-coefficients)
on the closure  $\ov \Gr^\l_G \subset \Gr_G$, and the pushforward 
of the constant sheaf $\un E_{\wt \Gr^\l_G}$ on the resolution $\nu^{\l}: \wt \Gr^\l_G \to \ov \Gr_G^\l$,  are all canonically $\Aut^0(\cO)\ltimes G(\cO)$-equivariant.
\end{ex}

The following is  easy to observe by induction on the poset $\Lambda_T^+$, and the observation that the stabilizer
 within $\Aut^0(\cO)\ltimes G(\cO)$
of the point $t^\l G(\cO)\in\Gr_G$, for $\l\in \Lambda_T^+$, is connected.

\begin{lemma}\label{l:gen Sat} Every object of $\Sat_G$, respectively   $\Sat^0_G$, is isomorphic to a finite complex of objects from each of the following collections
 \begin{equation*}
\xymatrix{
\{\un E_{\Gr^\l_G}\}_{\l \in \Lambda^+_T}
&
\{\un E_{\ov \Gr^\l_G}\}_{\l \in \Lambda^+_T}
&
\{\IC^\l\}_{\l \in \Lambda^+_T}
&
\{\nu^{\l}_{!}\un E_{\wt \Gr^\l_G}\}_{\l \in \Lambda^+_T}
}
\end{equation*}
\end{lemma}

\begin{proof}
For the first three collections, the assertion is standard and  the proof is left to the reader. 

Let $\Sat_G^{\le \l}$ (resp. $\Sat^{<\l}_{G}$) be the full dg-subcategory of $\Sat_{G}$ consisting of objects supported on $\ov\Gr^{\l}_{G}$ (resp. supported on $\ov\Gr^{\l}_{G}\bs \Gr^{\l}_{G}$). We show by induction on $\l\in\L_{T}^{+}$ that every object in $\Sat^{\le\l}_G$ is a finite complex of $\{\nu^{\mu}_{!}\un E_{\wt \Gr^\mu_G}\}_{\mu\le \l}$. This is trivial for $\l=0$. Suppose this  is true for all $\l'<\l$.  Clearly, every object in $\Sat_G^{\le \l}$ is a finite complex of $\un{E}_{\Gr^{\l}_{G}}$ and objects in $\Sat^{<\l}_{G}$. Recall that $\nu^{\l}: \wt \Gr^\l_G \to \ov \Gr_G^\l$ is a resolution, so in particular an isomorphism over the open dense locus $\Gr^\l_G \subset \ov \Gr^\l_G$. Thus the cone of the natural map $\un E_{\Gr^\l_G} \to \nu^{\l}_{!}\un E_{\wt \Gr^\l_G}$ is supported on 
$ \ov \Gr^\l_G \setminus \Gr^\l_G$, hence an object in $\Sat^{<\l}_{G}$. Therefore  every object in $\Sat_G^{\le \l}$ is a finite complex of $\nu^{\l}_{!}\un E_{\wt \Gr^\l_G}$ and $\Sat^{<\l}_{G}$. Finally, since $\ov\Gr^{\l}_{G}\bs \Gr_{G}^{\l}=\cup_{\mu<\l}\ov\Gr^{\mu}_{G}$, by inductive hypothesis, every object in $\Sat^{<\l}_{G}$ is a finite complex of $\{\nu^{\mu}_{!}\un E_{\wt \Gr^\mu_G}\}_{\mu<\l}$. Therefore, every object in $\Sat^{\le\l}_G$ is a finite complex of $\{\nu^{\mu}_{!}\un E_{\wt \Gr^\mu_G}\}_{\mu\le \l}$.

%Recall that $\nu^{\l}: \wt \Gr^\l_G \to \ov \Gr_G^\l$ is a resolution, so in particular an isomorphism over the open dense locus $\Gr^\l_G \subset \ov \Gr^\l_G$. Thus the cone of the natural map $\un E_{\Gr^\l_G} \to \nu^{\l}_{!}\un E_{\wt \Gr^\l_G}$ is supported on 
%$ \ov \Gr^\l_G \setminus \Gr^\l_G$. The assertion now follows by induction on the poset $\Lambda_T^+$.

The argument  in the case of $\Sat^{0}_{G}$ is entirely the same.
\end{proof}

%
%Under  the geometric Satake correspondence, 
%for each $\lambda\in \Lambda_T^+$, the irreducible $\dG$-representation $V^\l$ corresponds to the intersection complex $\IC^\l$ of the $G(\cO)$-orbit closure $\ov \Gr_G^\l\subset \Gr_G$. 

%
%
%Let $\Sat^0_G = \Sh_c((\Aut^0(\cO) \ltimes G(\cO))\backslash \Gr_G)$ be the dg category of $ \Aut^0(\cO) \ltimes G(\cO)$-equivariant constructible complexes on $\Gr_G$
%with compact support. 
%Note each intersection complex $\IC^\l$ is canonically equivariant for the natural $\Aut^0(\cO)\ltimes G(\cO)$-action.

%%%%%%%%%%%%%%%%%%%%%%%%%%%%%%%%%%%%%%%%%%%%%%%%%%%%
%%%%%%%%%%%%%%%%%%%%%%%%%%%%%%%%%%%%%%%%%%%%%%%%%%%%
%%%%%%%%%%%%%%%%%%%%%%%%%%%%%%%%%%%%%%%%%%%%%%%%%%%%

\section{Constructions over a curve}

Let $X$ be a connected smooth projective curve over $\CC$.

%%%%%%%%%%%%%%%%%%%%%%%%%%%%%%%%%%%%%%%%%%%%%%%%%%%%%%
%%%%%% red starts %%%%%%%%%%%%
%{\color{red}

 \subsection{Local coordinates}
For a $\CC$-algebra $R$ and  an $R$-point $x\in X(R)$ with graph $\Gamma_{x}\subset X_{R}=X\times_{\CC}\Spec R$, 
let $\wh\cO_{x}$ be the completion of $X_{R}$ along the ideal $\cI_{x}$ defining the graph $\Gamma_{x}$. Let $\wh{\cI}_{x}=\cI_{x}\wh\cO_{x}$. Let $D_{x}=\Spec \wh\cO_{x}$ be the formal disc around $\Gamma_{x}$, and $D^\times_x = D_{x}\setminus\Gamma_{x}$ be the formal punctured disk.
  
Consider the presheaf $\Coord^0(X)^{pre}$ on affine $\CC$-algebras whose value at a $\CC$-algebra $R$ is the set of pairs $(x,t_{x})$ where $x\in X(R)$ and $t_{x}\in \wh{\cI}_{x}$ that induces an $R$-linear isomorphism $\ph_{t_{x}}: R[[t]]\isom\wh{\cO}_{x}$ sending $t$ to $t_{x}$.  The sheafification of $\Coord^0(X)^{pre}$ is representable by a right $\Aut^0(\calO)$-torsor $\Coord^0(X)\to X$. 
%its points are pairs $(x,\a_{x})$ where $x\in X$ is a point, and $\a_{x}:\cO\stackrel{\sim}{\to} \cO_{x}$ is an isomorphism preserving the maximal ideals. Alternatively, we may denote a point in $\Coord^{0}(X)$ by a pair $(x,t_{x})$ where we set $t_{x}=\a_{x}(t)$.

\subsubsection{The $\Aut(\cO)$-action}\label{sss:AutO} The $\Aut^0(\calO)$-action on $\Coord^0(X)$ extends to an
$\Aut(\calO)$-action: for an automorphism $\a$ of $R[[t]]$ and $(x,t_{x})\in \Coord^0(X)(R)$, $(x,t_{x})\cdot\a=(x', t_{x'})$ where $x'$ is the $R$-point of $X$ defined by the ring homomorphism
\begin{equation*}
\xymatrix{     \wh{\cO}_{x}\ar[r]^{\a^{-1}\circ\ph_{t_{x}}^{-1}} & R[[t]]\ar[r]^{t\mapsto 0} & R,
}
\end{equation*}
and $t_{x'}=\ph_{t_{x}}(\a(t))$. The induced infinitesimal $\Der(\calO)$-action is simply transitive, i.e., the moment map $T^{*}\Coord^{0}(X)\to \Der(\cO)^{*}$ restricts to an isomorphism on each cotangent space
\begin{equation}\label{eq:cot coord}
\xymatrix{ T^{*}_{(x,t_{x})}\Coord^{0}(X)\ar[r]^-{\sim} & \Der(\cO)^{*}, & \textup{ for }(x,t_{x})\in \Coord^{0}(X).
}
\end{equation}

\subsection{Uniformization}\label{sss:unif}

We will have use for an infinitesimal formula for the Grassmannian uniformization of $\Bun_G(X)$.
Consider the natural map
\begin{equation*}
\xymatrix{
u:%G(\cK) \times \Coord^0(X) \ar[r] & 
\Gr_G \times \Coord^0(X) \ar[r] & \ar[r] \Gr_{G, X} \ar[r] & \Bun_G(X).
}
\end{equation*}
To a $\CC$-algebra $R$, a point $g(t)G(\cO\wh{\otimes}R) \in \Gr_G(R)$, and a point $(x, t_x) \in \Coord^0(X)(R)$ ($t_x$ is a formal coordinate along the graph $\Gamma_x$ of $x$), it assigns the $G$-bundle $\calE = u(g(t)G(\cO), x, t_x)$ on $X_R$ given by glueing the trivial bundles over $D_{x}$ and $X_R\setminus \Gamma_x$ by the transition matrix $g(t_x)$. By construction, $\cE$ is equipped with a trivialization $\tau_{X\setminus x}: \cE|_{X_R\setminus \Gamma_x}\simeq G\times (X_R\setminus \Gamma_x)$.

\begin{lemma}\label{l:Aut inv} The  map $u$ is invariant under the anti-diagonal action of $\Aut(\cO)$ on $\Gr_{G}\times\Coord^{0}(X)$. 
%In other words,
%it factors through a diagram
%\begin{equation*}
%\xymatrix{
%\Gr_{G, X}  \ar[r] & \Gr_{G, X_\dR} \ar[r] &   \Bun_{G}(X).
%}
%\end{equation*}
%where  we denote by
%\begin{equation*}
%\xymatrix{
%\Gr_{G, X_\dR}  = \Gr_G \overset{\Aut(\calO)}{\times} \Coord^0(X) \ar[r] & X_\dR
%}\end{equation*}
% the crystal associated to the  connection on $\Gr_{G, X}\to X$.
\end{lemma}

\begin{proof} Let $R$ be a $\CC$-algebra. For a point $(g, x,t_{x})\in G(R((t)))\times \Coord^{0}(X)(R)$, and an automorphism $\a$ of $R[[t]]$, we have $\a\cdot g\in G(R((t)))$ is given by the composition $\Spec R((t))\xr{\a_{*}} \Spec R((t))\xr{g} G$, and $(x,t_{x})\cdot \a=(x', t_{x'})$ is described in \S\ref{sss:AutO}. Note that  $D_{x}=D_{x'}$, $D_{x}^{\times}=D_{x'
}^{\times}$ and $X_{R}\setminus\Gamma_{x}=X_{R}\setminus\Gamma_{x'}$. The $G$-bundle $u(\a^{-1}\cdot g, (x,t_{x})\cdot\a)$ on $X_{R}$ is obtained by gluing the trivial bundles over $D_{x'}=D_{x}$ and $X_R\setminus \Gamma_{x'}=X_{R}\setminus\Gamma_{x}$ by the following transition matrix
\begin{equation*}
\xymatrix{ D_{x'}=D_{x}\ar[rr]^-{(\ph_{t_{x}}\circ \a)_{*} } && \Spec R((t)) \ar[rr]^-{\a^{-1}\cdot g}  && G.
}
\end{equation*}
Direct calculation shows that the two appearances of $\a$ cancel out, and the composition above is the same as $g(\ph_{t_{x}}(t))=g(t_{x}):D_{x}\to G$, which is the same transition matrix defining the $G$-bundle $u(g(t), x,t_{x})$. This proves that $u$ is invariant under $\Aut(\cO)$.
%Recall that $\Gr_{G, X}$ classifies a point $x\in X$, a $G$-bundle $\calE$ over $X$
%with a trivialization of the restriction $\calE|_{X\setminus x}$. The map $\Gr_{G, X}\to \Bun_G(X)$
% remembers $\calE$. 
%
%Observe that $\Gr_{G, X_\dR}$ classifies a point $x_\dR\in X_\dR$, a $G$-bundle $\calE$ over $X$
%with a trivialization of the restriction $\calE|_{X\setminus \wh x_\dR}$ where we write $\wh x_\dR \subset X$ for the inverse image of $ x_\dR\in X_\dR$. If we choose a point $x\in X$ with image $ x_\dR \in X_\dR$, then $\wh x_\dR\subset X$ is  the formal neighborhood of $x\in X$. The map $\Gr_{G, X}\to \Gr_{G, X_\dR}$ takes $x\in X$ to its image $x_\dR\in X_\dR$,
%and restricts the trivialization of $\calE$ along the map $X\setminus \wh x_{\dR} \incl X\setminus x$. 
%
%With the above descriptions, the asserted factorization  is evident.
\end{proof}

%}
%%%%%%% red ends %%%%%%%%%%%%%%%%%%%%%%%%

%%%%%%%%%%%%%%%%%%%%%%%%%%%%%%%%%%%%%%%%%%%%%%%%%%%%

\subsection{Higgs fields}

Let $\TT^*\Bun_G(X)$ be the total space of the cotangent complex of $\Bun_G(X)$. 
It is a derived algebraic stack locally of finite type.
Its fiber at $\calE\in \Bun_G(X)$ is given by the complex of derived sections
\begin{equation*}
\xymatrix{
\TT^*_\calE \Bun_G(X) \simeq \Gamma(X, \frg^*_\calE \otimes\om_{X})
}
\end{equation*}
where $\frg^*_\calE = \frg^*\times_G \calE$ denotes the coadjoint bundle of $\calE$, and $\om_X$ the canonical bundle of $X$.

We will  be exclusively interested in the traditional cotangent bundle $T^*\Bun_G(X)$ given by the underlying classical stack
of  $\TT^*\Bun_G(X)$. It is an algebraic stack locally of finite type.
Its fiber at $\calE\in \Bun_G(X)$ is given by the space of Higgs fields
\begin{equation*}
\xymatrix{
T^*_\calE \Bun_G(X) \simeq H^0(X, \frg^*_\calE \otimes\om_{X}).
}
\end{equation*}

%%%%%%%%%%%%%%%%%%%%%%%%%%%%%%%%%%%%%%%%%%%%%%%%%%%%

\subsubsection{Global nilpotent cone}

Consider the characteristic polynomial map 
\begin{equation*}
\xymatrix{
\chi:\frg^* = \Spec \Sym (\frg) \ar[r] & \Spec \Sym (\frg)^G \simeq \Spec \Sym (\frh)^W =: \frc
}
\end{equation*}
and recall it is  $G \times \GG_m$-equivariant where $G$ acts trivially on $\frc$.

Let $ A_G(X) := H^0(X, \frc_{\omega_X})$ be the Hitchin base, 
where $\frc_{\omega_X} = \dot\omega_X\twtimes{\GG_m} \frc$ denotes the associated bundle,
where we write $\dot\om_{X}$ for the  $\Gm$-torsor associated with the line bundle $\om_{X}$.

Introduce the Hitchin map
\begin{equation*}
\xymatrix{
\Hitch:T^* \Bun_G(X) \ar[r] & A_G(X) 
&
\Hitch(\calE, \phi) = \chi(\phi).
}
\end{equation*}

The global nilpotent cone is the inverse image of the trivial point
\begin{equation*}
\xymatrix{
\calN_G(X) = \Hitch^{-1}(0) \subset T^* \Bun_G(X).
}
\end{equation*}
It is a Lagrangian substack by \cite{L,Gin} in the sense that for a  smooth map $u:U\to \Bun_G(X)$ from a scheme $U$,
with induced correspondence
\begin{equation*}
\xymatrix{
T^*U & \ar[l]_-{du}  T^*\Bun_G(X) \times_{ \Bun_G(X)} U \ar[r]^-{u_{\nat}} & T^*\Bun_G(X),
}
\end{equation*}
 we obtain a Lagrangian subvariety 
\begin{equation*}
\xymatrix{
du(u_{\nat}^{-1}(\calN_G(X))) \subset T^*U.
}
\end{equation*}
For $\calE\in \Bun_G(X)$, the fiber of $\calN_G(X)$ over $\cE$ is given by the space of  nilpotent Higgs fields
\begin{equation*}
\xymatrix{
\phi:X\ar[r] &  \calN_{\calE \times \om_{X}}
}
\end{equation*}
 where $\calN\subset \frg^*$ denotes the traditional nilpotent cone with the coadjoint action by $G$ and dilation by $\Gm$,
and $ \calN_{\calE \times \om_{X}} = (\calE \times_{X} \dot\omega_X) \twtimes{G \times \GG_m} \cN$ denotes the associated bundle, where as above we write $\dot\om_{X}$ for the  $\Gm$-torsor associated with the line bundle $\om_{X}$.

Note since
$\calN\subset  \frg^*$ is closed, a Higgs field $\phi$ is nilpotent if and only if it is generically nilpotent.

\subsubsection{The differential of the uniformization}\label{sss:du}

For a $\CC$-point $(g(t)G(\cO), x,t_x)\in \Gr_G\times\Coord^0(X)$ with image $\cE\in \Bun_G(X)$ under $u$,  the differential $du$ induces a map on cotangent spaces
\begin{equation*}
\xymatrix{
du :  T^*_\calE\Bun_G(X) \ar[r] & T_{g(t)G(\cO)}^*\Gr_G \times T^*_{(x, t_x)}\Coord^0(X) 
}
\end{equation*}
which under our identifications \eqref{eq:cot coord} and \eqref{eq:cot Gr} is more concretely a map
\begin{equation*}
\xymatrix{
H^0(X, \frg^*_{\calE}\otimes\om_X)  \ar[r] &  \Ad_{g(t)}\frg^{*}(\cO) dt \times \Der(\cO)^*.
}
\end{equation*}
Here the first component is the composition of the trivialization $\tau_{X\setminus x}$, restriction to $D^{\times}_{x}$ and the change of variable $t_{x}\mapsto t$:
\begin{equation}\label{rest Dx}
\xymatrix{H^0(X, \frg^*_{\calE}\otimes\om_X)\ar[r] & H^{0}(X\setminus x, \frg^{*}\otimes\om_{X})\ar[r] & \frg^{*}(\cK_{x})dt_{x}\ar[r]^{t_{x}\mapsto t}_{\sim} & \frg^{*}(\cK)dt
}
\end{equation}
and its image lies in $\Ad_{g(t)}\frg^{*}(\cO) dt$. We denote the composition of the maps in \eqref{rest Dx} by $\phi\mapsto \phi|_{D^{\times}}$.
 
\begin{prop}\label{p:diff Gr BunG}
The second component of $du$ viewed as a bilinear pairing
\begin{equation*}
\xymatrix{
\Der(\cO) \times H^0(X, \frg^*_{\calE}\otimes\om_X)  \ar[r] &   \CC
}
\end{equation*}
%factors through the restriction 
%\begin{equation*}
%\xymatrix{
%\Der(\cO) \oplus H^0(X, \frg^*_{\calE}\otimes\om_X) \ar[r] & 
%\Der(\cO) \oplus \frg(\cK)dt  \ar[r] &  k
%}
%\end{equation*}
 takes the form
\begin{equation}\label{dr}
\xymatrix{
(f(t)\partial_{t}, \phi)\ar@{|->}[r] &  \Res_{t=0}(\jiao{f(t)g'(t)g(t)^{-1},\phi|_{D^\times}})
}
\end{equation}
where $\phi|_{D^\times}$ is defined as in \eqref{rest Dx}.
%the transport of  $\phi|_{D^\times_x}$ via the parameter $t_x$.
\end{prop}

\begin{proof}
By Lemma \ref{l:Aut inv}, the map $u$ factors through the anti-diagonal $\Aut(\cO)$-action on $\Gr_{G} \times \Coord^0(X)$.
Thus the  image of the differential $du$ lies in the kernel of the moment map
\begin{equation*}
\xymatrix{
\mu = \mu_1 \times (-\mu_2): T^{*}_{g(t)G(\cO)}\Gr_G \times T^{*}_{(x,t_x)}\Coord^0(U)\ar[r] & \Der(\calO)^{*}
}
\end{equation*}
The moment map $\mu_2$ is our usual identification \eqref{eq:cot coord}. Using this identification, we thus have
\begin{equation*}
\xymatrix{
\ker(\mu) = \{(\eta, \mu_{1}(\eta)) \in T^{*}_{g(t)G(\cO)}\Gr_G  \times \Der(\cO)^* \, \vert\,  \eta\in T^{*}_{g(t)G(\cO)}\Gr_G\}
}
\end{equation*}

Recall from~\eqref{eq: moment map},
the first factor $\mu_1$, viewed as a bilinear map, takes the form
\begin{equation*}
\xymatrix{
\Der(\cO)\times   T^*_{g(t)G(\cO)} \Gr_G \ar[r] & \CC
}
\end{equation*}
\begin{equation*}
\xymatrix{
(f(t)\partial_t, \eta) \ar@{|->}[r] &\langle f(t) \partial_t, \mu_1(\eta)\rangle
=   \Res_{t=0} \langle f(t)g'(t) g(t)^{-1}, \eta\rangle
}
\end{equation*}
Using the fact that the first component of $du$ is given by $\phi\mapsto \phi|_{D^{\times}}$, we conclude  that  the second component of $du$, viewed as a bilinear map, takes the asserted form
\begin{equation*}
\xymatrix{
(f(t)\partial_t, \phi) \ar@{|->}[r] &  \langle f(t) \partial_t, \mu_1(\phi|_{D^\times})\rangle
=  \Res_{t=0} \langle f(t)g'(t) g(t)^{-1}, \phi|_{D^\times}\rangle.
}
\end{equation*}
\end{proof}

%%%%%%%%%%%%%%%%%%%%%%%%%%%%%%%%%%%%%

\subsection{Hecke modifications}

Introduce the Hecke diagram 
\begin{equation*}
\xymatrix{
\Bun_G(X)  & \ar[l]_-{p_-} \Hecke_G(X) \ar[r]^-{p_+ \times \pi} &   \Bun_G(X) \times X  \\
 }
\end{equation*}
where $\Hecke_G(X)$ classifies a point $x\in X$, a pair of bundles $\calE_-, \calE_+ \in \Bun_G(X)$, together with
an isomorphism  of $G$-bundles over $X \setminus x$
\begin{equation*}
\xymatrix{
\a: \calE_-|_{X \setminus x} \simeq \calE_+|_{X \setminus x}.
 }
\end{equation*}
The map $p_-$ returns the bundle $\calE_-$, $p_+$ the bundle $\calE_+$, and $\pi$ the point $x$.

For each $\lambda\in \Lambda_T^+$, we have a subdiagram
\begin{equation}\label{eq:Hk lam}
\xymatrix{
\Bun_G(X)  & \ar[l]_-{p^\l_-} \Hecke^\l_G(X) \ar[r]^-{p^\l_+\times \pi} &   \Bun_G(X) \times X  \\
 }
\end{equation}
where  $\calE_-, \calE_+$ are in relative position $ \l$ at the point  $x$, i.e., upon trivializing $\cE_{-}$ and $\cE_{+}$ over $D_{x}$ and choosing a local coordinate $t_{x}$ at $x$ to identify $D_{x}$ with $D$, $\a$ as an element in $G(\cK)$ lies in $G(\cK)^{\l}$. The closure  $ \ov\Hecke^\l_G(X) $ of $\Hecke^\l_G(X) $ also gives a diagram
\begin{equation}\label{eq:ov Hk lam}
\xymatrix{
\Bun_G(X)  & \ar[l]_-{\ov p^\l_-} \ov\Hecke^\l_G(X) \ar[r]^-{\ov p^\l_+\times \pi} &   \Bun_G(X) \times X  \\
 }
\end{equation}
where  $\calE_-, \calE_+$ are in relative position $ \leq \l$ at the point  $x$, i.e., upon the same trivializations as above, $\a$ as an element in $G(\cK)$ lies in $\ov{G(\cK)}^{\l}$.

\subsubsection{Satake kernels}\label{sss:Sat kernel}

%%%%%% red starts %%%%%%%%
%{\color{red}  
Using the $\Aut^{0}(\cO)$-action on $G(\cO)$ and $G(\cK)$, we introduce the group scheme
\begin{equation*}
\xymatrix{ \cG_{X}^{\cO}=\Coord^{0}(X)\overset{\Aut^{0}(\calO)}{\times}G(\cO) \ar[r] & X
}
\end{equation*}
and the group ind-scheme 
\begin{equation*}
\xymatrix{ \cG_{X}=\Coord^{0}(X)\overset{\Aut^{0}(\calO)}{\times}G(\cK)  \ar[r] & X
}
\end{equation*}
%}
%%%%%%%% red ends %%%%%%%%%%

%Let $\DD= D \coprod_{D^{\times}} D$ denote the non-separated disk with two zeros $0_+, 0_-\in \DD$.
Let $\wh\Bun_G(X)_X \to \Bun_G(X)\times X$ denote the $\calG^\cO_X$-torsor 
classifying a point $x\in X$, a 
$G$-bundle $\calE$ over $X$, and a trivialization of the restriction $\calE|_{D_x}$. By \cite[2.8.4]{BD},  the $\calG^\cO_X$-action on $\wh\Bun_G(X)_X$ (by changes of trivialization) naturally extends to a $\cG_X$-action by the usual gluing paradigm.

We have a canonical isomorphism
\begin{equation}\label{Hk Gr}
\xymatrix{
\Hecke_G(X)\simeq  \wh \Bun_G(X)_X \overset{\cG^\cO_X}{\times_X} \Gr_{G, X}
}
\end{equation}
so that $ p_-$ is the evident projection to the first factor, and $ p_+$  is given by the $\cG_{X}$-action on $\wh \Bun_G(X)_X $.

Consider the resulting natural diagram
\begin{equation*}
\xymatrix{
\Hecke_G(X)  & \ar[l]_-q \Hecke_G(X) \times_X \Coord^0(X) \ar[r]^-p &   G(\cO)\backslash \Gr_G
 }
\end{equation*}
where $q$ is the evident $\Aut^0(\cO)$-torsor, and $p$ records the relative position of the pair $(\cE|_{D_{x}},\cE'|_{D_{x}})$ using the local coordinate at $x$.

The $\Aut^{0}(\cO)$-equivariance of any object $\cV\in \Sat^0_G$ induces an $\Aut^{0}(\cO)$-equivariance on 
the pullback $p^{*}\cV$ 
along the $\Aut^0(\cO)$-equivariant map $p$. Since $q$ is an $\Aut^0(\cO)$-torsor,  the $\Aut^{0}(\cO)$-equivariant complex $p^*\cV$ descends along $q$ to a unique complex we denote by $\cV'\in \Sh(\Hecke_{G}(X), E)$.

Let $\Sh(\Bun_G(X), E)$ be the dg derived category of all complexes of $E$-modules on $\Bun_G(X)$, in the sense explained in Section~\ref{ss:auto notation}.

Introduce  the  Hecke functor
\begin{equation}\label{eq:Sat kernel}
\xymatrix{
 H_\cV:\Sh(\Bun_G(X), E) \ar[r] & \Sh(\Bun_G(X) \times X, E) 
&
 H_\cV (\calF) =   ( p_+ \times \pi)_! (\cV' \otimes_{ E} ( p_{-})^*\calF). 
}
\end{equation}

\begin{ex}
For $\cV = \un E_{\Gr^\l_G}$ with its natural $ \Aut^0(\cO) \ltimes G(\cO)$-equivariance, the corresponding Hecke functor $H_\cV=H^\l$ is given by (using notation from \eqref{eq:Hk lam})
\begin{equation*}
\xymatrix{
H^\l:   \Sh(\Bun_G(X), E) \ar[r] & \Sh(\Bun_G(X) \times X, E)}
\end{equation*}
\begin{equation*}
\xymatrix{
H^\l(\calF) = (p^\l_{+}\times\pi)_!(p^\l_{-})^*\calF.}
\end{equation*}
%For $\cV = \un E_{\ov \Gr^\l_G}$ with its natural $ \Aut^0(\cO) \ltimes G(\cO)$-equivariance, we have $ H_\cV \simeq \ov H^\l$.
\end{ex}

%%%% 

\subsubsection{Hecke stack resolutions}
The maps $p_-^\l, p_+^\l\times \pi$ are in general smooth and (locally on the base) quasi-projective
but not proper, while the maps
$\ov p_-^\l, \ov p_+^\l\times \pi$ are  in general (locally on the base) projective but not smooth.
Thus estimating the singular support of the Hecke functors is not as concrete as we would like. 
We will find it convenient in Section~\ref{s: main thm} to work with a smooth resolution to address this.

Under the  isomorphism \eqref{Hk Gr}, we have
\begin{equation*}
\xymatrix{
\ov\Hecke^\l_G(X)\simeq  \wh \Bun_G(X)_X \overset{\cG^\cO_X}{\times_X}\ov \Gr^\l_{G, X}.
}
\end{equation*}

Recall the resolution $\nu^{\l}: \wt \Gr_G^\l \to \ov\Gr_G^\l$ from Section ~\ref{s:Gr res}, and consider its globalization
\begin{equation*}
\xymatrix{
\wt \Gr_{G, X}^\l = \Coord^0(X) \overset{\Aut^0(\calO)}{\times} \wt \Gr^\l_G.
}
\end{equation*} 
Introduce the resolved Hecke stack
\begin{equation*}
\xymatrix{
r^\l:\wt \Hecke^\l_G(X) :=   \wh\Bun_G(X)_X \overset{\cG^\cO_X}{\times_X}\wt \Gr^\l_{G, X}
\ar[r] & 
 \wh\Bun_G(X)_X \overset{\cG^\cO_X}{\times_X}\ov \Gr^\l_{G, X}\simeq \ov\Hecke^\l_G(X).
}
\end{equation*}

Form
the resolved Hecke diagram
\begin{equation}\label{eq:wt Hk lam}
\xymatrix{
&& \ar[lld]_-{ \wt p^\l_-}  \wt \Hecke^\l_G(X) \ar[d]^-{r^\l} \ar[rrd]^-{\wt p_+^\l\times \pi}  &&  \\
\Bun_G(X)  && \ar[ll]_-{ \ov p^\l_-}  \ov \Hecke^\l_G(X)  \ar[rr]^-{\ov p_+^\l\times \pi} &&   \Bun_G(X) \times X  
 }
\end{equation}
with commutative triangles.
%The maps of \eqref{eq: resolved hecke diag} are the respective compositions 
%$\wt p^\l_- = p^\l_- \circ r$ and $\wt p_+^\l \times \pi= (p^\l_+ \times \pi) \circ r$  making the diagram commute.
Since  $\nu^{\l}: \wt \Gr_G^\l \to \ov\Gr_G^\l$ is $G(\cO)$-equivariant, its restriction over each $G(\cO)$-orbit $\Gr_G^\mu\subset \ov\Gr_G^\l$ is an \'etale locally trivial fibration, and therefore $r^\l$ is an \'etale locally trivial fibration above each smooth substack $\Hecke^\mu_G(X) \subset \ov \Hecke^\l_G(X)$, for $\mu\leq \lambda$.

Introduce the  Hecke functors
\begin{equation}\label{wt H}
\xymatrix{
\wt H^\l:\Sh(\Bun_G(X), E) \ar[r] & \Sh(\Bun_G(X) \times X, E) 
&
\wt H^\l (\calF)=  (\wt p_+^\l\times \pi)_! (\wt p^\l_{-})^*\calF.
}
\end{equation}
Then under the notation \eqref{eq:Sat kernel}, we have
\begin{equation*}
\xymatrix{   \wt H^\l     \simeq     H_\cV, &  \textup{ for } \cV = \nu^{\l}_!\un E_{\wt \Gr^\l_G}.
}
\end{equation*}

%\begin{remark}\label{rem: upper tri res}
%The cone of the natural map $\un E_{\ov \Gr^{\l}_{G}}\to  r_!\un E_{\wt \Gr^\l_G}$ is a
%finite complex of the collection of kernels $\{\un E_{\Gr^{\mu}_{G}}\}_{\mu <  \l}$.
% Therefore the cone of the induced natural transformation $\ov H^\l \to \wt H^\l$ is a finite complex of the collection of 
%Hecke functors $\{H^\mu\}_{\mu <  \l}$.
%\end{remark}

%%%%%%%%%%%%%%%%%%%%%%%%%%%%%%%%%

%
%\marginpar{may be removed.}
%
%\begin{remark}\label{rem: Z(X)}
%With the above constructions in hand, we can now give a precise construction of $Z^\l(X)$ and diagram~\eqref{eq:extended hecke} as mentioned in Remark~\ref{rem: Z(x)}.
%
%Denote by $\wt\Bun_G(X)_X \to \Bun_G(X)\times X$ the $\calG^\cO_X$-torsor 
%classifying a point $x\in X$, a 
%$G$-bundle $\calE$ over $X$, and a trivialization of the restriction $\calE|_{D_x}$.
%Then we have an  isomorphism
%\begin{equation}
%\xymatrix{
% \Hecke^\l_G(X) \simeq  \Gr^\l_{G, X}\overset{\cG^\cO_X}{\times_X} \wt \Bun_G(X)_X
% }
%\end{equation}
%and can define $Z^\l(X)$ similarly to be the associated bundle
%\begin{equation}
%\xymatrix{
% Z^\l(X) =  Z^\l_X \overset{\cG^\cO_X}{\times_X} \wt \Bun_G(X)_X
% }
%\end{equation}
%where we set $Z^\l_X = \Coord(X) \overset{\Aut(\calO)}{\times} Z^\l$  following our usual notation.
%The maps $i^\l$ and $q^\l\times \pi$ of diagram~\eqref{eq:extended hecke}  are the evident ones. 
%
%\end{remark}

%%%%%%%%%%%%%%%%%%%%%%%%%%%%%%%%%%%%%%%%%%%%%%%%%%%%
%%%%%%%%%%%%%%%%%%%%%%%%%%%%%%%%%%%%%%%%%%%%%%%%%%%%
%%%%%%%%%%%%%%%%%%%%%%%%%%%%%%%%%%%%%%%%%%%%%%%%%%%%

\section{Microlocal geometry}\label{s: micro}

%%%%%%%%%%%%%%%%%%%%%%%%%%%%%%%%%%%%%%%%%%%%%%%%%%%%

\subsection{Singular support}

We recall some basic definitions and properties of the singular support of a complex of sheaves.
The standard reference is Kashiwara-Schapira's book~\cite{KS}. 
Much of the theory developed therein is for bounded or bounded below complexes. One can remove this assumption, using the formalism of six operations presented in~\cite{Sp} and the specific microlocal foundations provided by~\cite{RS}.

%%%%%%%%%%%%%%%%%%%%%%%%%%%%%%%%%%%%%%%%%%%%%%%%%%%%

\subsubsection{Schemes}\label{sss:sing scheme}

Let $U$ be a smooth scheme with cotangent bundle $T^*U$. 

Let $\Sh(U)$ 
be the dg derived category of all complexes of abelian groups on $U$.
We will often abuse terminology and use the term sheaves
to refer to its objects.

Suppose $\calS = \{U_\alpha\}_{\alpha\in A}$ is a $\mu$-stratification of $U$ in the sense of \cite[Def. 8.3.19]{KS}.
Let $\Lambda_\cS = \cup_{\alpha\in A} T^*_{U_\alpha} U \subset T^* U$ denote the union of the conormal bundles to the strata.

Let $\Sh_{\cS}(U) \subset \Sh(U)$ denote the full dg subcategory of complexes locally constant along the strata of $\cS$.
For  any $\cF\in \Sh_{\cS}(U)$, its singular support $\sing(\calF)\subset T^*X$ is a closed conic  Lagrangian subscheme.
By~\cite[Prop. 8.4.1]{KS}, we have the containment $\sing(\cF) \subset \Lambda_\cS$, and hence $\sing(\cF)$ is a union of some irreducible components of $\Lambda_\cS$. An irreducible component of $\Lambda_\cS$ is not in the singular support $\sing(\cF)$ if and only if the vanishing cycles $\phi_f (\cF)$ are trivial for some germ of a  function $f$ at a point $u\in U$  whose  differential $df|_u \in T^* U$ is a generic point of the  irreducible component.

Given a closed conic Lagrangian subscheme $\Lambda\subset T^*U$, by~\cite[Cor. 8.3.22]{KS}, we may choose a $\mu$-stratification $\calS = \{U_\alpha\}_{\alpha\in A}$ of $U$
such that $\Lambda \subset \Lambda_\cS = \cup_{\alpha\in A} T^*_{U_\alpha} U \subset T^* U$.
Denote by $\Sh_\Lambda(U) \subset \Sh_\cS(U)$ 
the full dg subcategory 
of complexes with  singular support $\sing(\calF)$ contained within $\Lambda$.
By~\cite[Thm. 8.3.20, Prop. 8.4.1]{KS},  this is independent of the choice of $\mu$-stratification. 

Singular support satisfies the following functoriality. For details we refer to \cite[Chapter 5.4]{KS}. Let $f: U\to V$ be a morphism between smooth schemes. We consider the Lagrangian correspondence
\begin{equation*}
\xymatrix{T^{*}U  & \ar[l]_-{df} T^*V \times_{V} U   \ar[r]^-{f_{\nat}} & T^*V }
\end{equation*}

\begin{enumerate}
\item Smooth pullback. Suppose $f$ is smooth, then for $\cG\in \Sh(V)$,  we have
\begin{equation}\label{eq: pullback}
\xymatrix{         \sing(f^{*}\cG)=df(f_{\nat}^{-1}(\sing(\cG))).
}
\end{equation}
\item Proper pushforward. Suppose  $f$ is proper, then for $\cF\in \Sh(U)$, we have
\begin{equation}\label{eq: pushforward}
\xymatrix{\sing(f_{*}\cF)\subset f_{\nat}(df^{-1}(\sing(\cF)))}.
\end{equation}
\end{enumerate}

%%%%%%%%%%%%%%%%%%%%%%%%%%%%%%%%%%%%%%%%%%%%%%%%%%%%

\subsubsection{Stacks}\label{ss: sing stacks}

Let  $Y$ be a smooth stack, let $\TT^* Y$ be the total space of its cotangent complex,
and $T^*Y$ its underlying classical cotangent bundle.

Let $\Sh(Y)$ 
be the dg derived category of all complexes on $Y$.

Thanks to the  functoriality recalled above, the notion of singular support readily extends to this setting. 
Namely,
for $\calF \in \Sh(Y)$, we may assign its singular support $\sing(\calF) \subset T^*Y$ uniquely characterized
by the following property. For a  smooth map $u:U\to Y$ where $U$ is a smooth scheme,
with induced correspondence
\begin{equation*}
\xymatrix{
T^*U & \ar[l]_-{du} T^*Y \times_Y U \ar[r]^-{u_{\nat}} & T^*Y,
}
\end{equation*}
we have the compatibility
\begin{equation*}
\xymatrix{
\sing(u^*\calF) = du(u_{\nat}^{-1}(\sing(\calF)) \subset T^*U
}
\end{equation*}

Observe that with this characterization in hand, the functoriality recalled in Section~\ref{sss:sing scheme} readily extends to representable maps of smooth stacks.
%
%
%Given a conic Lagrangian subvariety $\Lambda\subset T^*Y$, 
%introduce 
%the full subcategory $\Sh_\Lambda(Y) \subset \Sh(Y)$ 
%of complexes with  $\sing(\calF) \subset \Lambda$.

%%%%%%%%%%%%%%%%%%%%%%%%%%%%%%%%%%%%%%%%%%%%%%%%%%%%

\subsection{Analysis of Hecke correspondences}\label{s: main thm}

Let $  \Sh_{\calN_G(X)}(\Bun_G(X), E) \subset \Sh(\Bun_G(X), E)$ denote the
 full subcategory 
of complexes with singular support in the global nilpotent cone 
\begin{equation*}
\xymatrix{
\sing(\calF) \subset \calN_G(X).
}
\end{equation*}

%We will analyze here the singular support of $\calF\in  \Sh_{\calN_G(X)}(\Bun_G(X), E)$ under the Hecke functors
%\begin{equation*}
%\xymatrix{
%H^\l:\Sh(\Bun_G(X), E) \ar[r] & \Sh(\Bun_G(X) \times X, E) 
%&
%H^\l (\calF)\simeq  (p_+^\l\times \pi)_! (p^\l_{-})^*\calF 
%}
%\end{equation*}
%associated to the Hecke correspondences \eqref{eq:Hk lam}, for $\l\in \L^{+}_{T}$.
%\begin{equation}
%\xymatrix{
%\Bun_G(X)  & \ar[l]_-{p^\l_-} \Hecke^\l_G(X) \ar[r]^-{p^\l_+\times \pi} &   \Bun_G(X) \times X  \\
% }
%\end{equation}

The rest of the section is devoted to the proof of the following.

\begin{theorem}\label{th:ss} For any kernel $\cV\in \Sat^{0}_{G}$, the Hecke functor $H_{\cV}$ preserves nilpotent singular support, and,  for sheaves with nilpotent singular support,
it does not introduce non-zero singular codirections along the curve. In other words, for $\cF\in \Sh(\Bun_{G}(X), E)$, 
\begin{equation*}
\xymatrix{
\sing(\calF) \subset \calN_G(X) \implies 
\sing(H^\l(\calF)) \subset \calN_G(X) \times X, 
}
\end{equation*}
where $X\subset T^*X$ denotes the zero-section.
\end{theorem}

\begin{remark}
The results and arguments to follow will not involve any  object $\calF$ in any specific way, but  devolve to the maximal possible singular support of
 $\calN_G(X)$ itself.
\end{remark}

\begin{remark}
Since the curve $X$ and group $G$ will not change, to simplify the notation in what follows, let us write $\Bun = \Bun_G(X)$, $\Hecke = \Hecke_G(X)$, 
and similarly for related spaces.
\end{remark}

\begin{proof}[Proof of Theorem \ref{th:ss}.]
Let $\l$ be a dominant coweight.  First, to the map $ p_-^\l$, we have the Lagrangian correspondence
\begin{equation*}
\xymatrix{
T^*\Bun  & \ar[l]_-{ (p^\l_-)_{\nat}}   T^*\Bun \times_{\Bun} \Hecke^\l \ar[r]^-{dp^\l_-} & T^*\Hecke^\l
 }
\end{equation*}

Second, to the map $p_+^\l\times \pi$, we have the Lagrangian correspondence
\begin{equation*}
\xymatrix{
 T^*\Hecke^\l & \ar[l]_-{d(p_+^\l \times \pi)} T^*(\Bun \times X) \times_{\Bun \times X} \Hecke^\l\ar[r]^-{(p_+^\l\times \pi)_{\nat}} & T^*(\Bun \times X)
 }
\end{equation*}

%%%%% red starts %%%%%%%%
%{\color{red}
Pretending $p^{\l}_+$ is proper, the functoriality of singular support suggests the consideration of the following conic Lagrangian in $T^*(\Bun\times X)$
\begin{equation*}
\xymatrix{
\sing^\l := (p_+^\l\times \pi)_{\nat}(d(p_+^\l \times \pi))^{-1}(dp^\l_-)(p^\l_-)^{-1}_{\nat} \calN_G(X) \subset T^*(\Bun \times X).
}
\end{equation*}

\begin{claim} For any $\l\in\L_{T}^{+}$, the conic Lagrangian substack $\sing^\l$ satisfies
\begin{equation*}
\xymatrix{
\sing^\l \subset \calN_G(X) \times X.
}
\end{equation*}
\end{claim}
The claim can be viewed as a naive version of the theorem: if $p^{\l}_+$ were always proper, the Claim would imply the theorem by the functoriality of singular support recalled in \eqref{eq: pullback} and \eqref{eq: pushforward}. We will turn to the proof of the Claim in a moment, but let us first see that it implies the theorem.
%}
%%%%%%%% red ends %%%%%%%%%

\begin{lemma}\label{l:claim to thm}
The Claim implies Theorem~\ref{th:ss}. 
\end{lemma}

\begin{proof} 
%Recall  from \eqref{ov H} the Hecke functors $\ov H^{\l}$ associated to the Hecke correspondences \eqref{eq:ov Hk lam}, for $\l\in \L^{+}_{T}$. By induction, following the observation of Remark~\ref{rem: upper tri}, the theorem is equivalent to the assertion:
%\begin{equation*}
%\xymatrix{
%\sing(\calF) \subset \calN_G(X) \implies 
%\sing(\ov H^\l(\calF)) \subset \calN_G(X) \times X , \quad  \forall \l\in  \L^{+}_{T}.
%}
%\end{equation*}

%\begin{equation}
%\xymatrix{
%\ov H^\l:\Sh(\Bun_G(X)) \ar[r] & \Sh(\Bun_G(X) \times X) 
%&
%\ov H^\l (\calF)\simeq  (\ov p_+^\l\times \pi)_! (\ov p^\l_{-})^*\calF 
%}
%\end{equation}

Recall from \eqref{wt H} the Hecke functors $\wt H^\l$ associated to the Hecke correspondences \eqref{eq:wt Hk lam} (corresponding to the kernel $\nu^{\l}_{!}\un E_{\wt\Gr^{\l}_{G}}$), for $\l\in \L^{+}_{T}$. By Lemma \ref{l:gen Sat}, the theorem is equivalent to the assertion:
\begin{equation*}
\xymatrix{
\sing(\calF) \subset \calN_G(X) \implies 
\sing(\wt H^\l(\calF)) \subset \calN_G(X) \times X, \quad \forall \l\in  \L^{+}_{T}. 
}
\end{equation*}

%\begin{equation}
%\xymatrix{
%\wt H^\l:\Sh(\Bun_G(X)) \ar[r] & \Sh(\Bun_G(X) \times X) 
%&
%\wt H^\l (\calF)\simeq  (\wt p_+^\l\times \pi)_! (\wt p^\l_{-})^*\calF 
%}
%\end{equation}

%
%By induction, it suffices to show the Claim implies
%\begin{equation}
%\xymatrix{
%\sing(\calF) \subset \calN_G(X) \implies 
%\sing((\wt p^\l_- \times \pi)_!(\wt p^\l_+)^*\cF ) \subset \calN_G(X) \times X 
%}
%\end{equation}
%

To the map $ \wt p_-^\l$, we have the Lagrangian correspondence
\begin{equation*}
\xymatrix{
T^*\Bun  & \ar[l]_-{ (\wt p^\l_-)_{\nat}}   T^*\Bun \times_{\Bun} \wt \Hecke^\l \ar[r]^-{d\wt p^\l_-} & T^*\wt \Hecke^\l
 }
\end{equation*}

To the map $\wt p_+^\l\times \pi$, we have the Lagrangian correspondence
\begin{equation*}
\xymatrix{
 T^*\wt \Hecke^\l && \ar[ll]_-{d(\wt p_+^\l \times \pi)} T^*(\Bun \times X) \times_{\Bun \times X} \wt \Hecke^\l\ar[rr]^-{(\wt p_+^\l\times \pi)_{\nat}} && T^*(\Bun \times X)
 }
\end{equation*}

The standard properties~\eqref{eq: pullback}, \eqref{eq: pushforward} imply
\begin{equation}\label{eq: sing under res hecke}
\xymatrix{
\sing( \wt H^\l(\cF) ) \subset  
(\wt p_+^\l\times \pi)_{\nat}(d(\wt p_+^\l \times \pi))^{-1}(d\wt p^\l_-)(\wt p^\l_-)^{-1}_{\nat} \sing(\cF).
}
\end{equation}
We will show that the right hand side lies in $\cN_{G}(X)\times X$.

Unwinding the definitions, the right hand side of \eqref{eq: sing under res hecke} comprises all elements 
\begin{equation*} 
((\cE_+, x), (\phi_+, \theta)) \in T^*(\Bun \times X)
\end{equation*}
 arising as follows:
 there is a point $\wt h\in  \wt \Hecke^\l$ with images 
\begin{equation*}
\xymatrix{
 \cE_- =\wt p_-^\l(\wt h), \cE_+ =\wt p_-^\l(\wt h)\in  \Bun,
 & x = \pi(\wt h)\in X 
}
\end{equation*}
 along with a covector 
\begin{equation*}
\xymatrix{
 \phi_- \in  T^*_{\cE_-} \Bun %&  \phi_+ \in  T^*_{h_+} \Bun & \theta \in T^*_x X
 }
\end{equation*}
satisfying the equation
\begin{equation}\label{eq: equality over res}
\xymatrix{
 d\wt p^\l_-(\cE_{-}, \phi_-) = d(\wt p^\l_+ \times \pi) ((\cE_{+}, x), (\phi_+, \theta)) \in T^*_{\wt h} \wt \Hecke^\l.
 }
\end{equation}

Now for some $\mu\leq \l$, we have
\begin{equation*}  
\xymatrix{
h = r^\l(\wt h) \in \Hecke^\mu \subset \ov\Hecke^\l.
}
\end{equation*}
Recall that  $r^\l$ restricts to an \'etale locally trivial fibration 
\begin{equation*}
\xymatrix{
r^\l|_\mu:\wt \Hecke^\l|_{\Hecke^\mu} \ar[r] & \Hecke^\mu. 
}
\end{equation*}
Thus we may choose a section of $r^\l|_\mu$ above the formal  neighborhood of $h$
passing through $\wt h$.
Pullback along this section  shows that  \eqref{eq: equality over res}  also implies the equation
\begin{equation}\label{eq: equality over hecke}
\xymatrix{
 d p^\mu_-(\cE_{-}, \phi_-) = d(p^\mu_+ \times \pi) ((\cE_{+}, x), (\phi_+, \theta)) \in T^*_{h} \Hecke^\mu.
 }
\end{equation}
But this is precisely the equation that exhibits
\begin{equation*} 
((\cE_{+}, x), (\phi_+, \theta)) \in \sing^\mu.
\end{equation*}

By the Claim, $\sing^\mu\subset \cN_{G}(X)\times X$, therefore $((\cE_{+}, x), (\phi_+, \theta))\in \cN_{G}(X)\times X$. Hence the right hand side of \eqref{eq: sing under res hecke} lies in $\cN_{G}(X)\times X$. Thus the Claim implies the theorem.
\end{proof}

Now it remains to prove the Claim.

\begin{proof}[Proof of the Claim.]
Fix $\CC$-points $x\in X$, $\calE_-, \calE_+ \in \Bun$, and respective covectors 
\begin{equation*}
\xymatrix{
\theta\in T^*_x X, & 
\phi_- \in H^0(X, \frg^*_{\calE_-}\otimes \om_X), & \phi_+ \in H^0(X, \frg^*_{\calE_+}\otimes \om_X).
}
\end{equation*}

Fix an isomorphism  
\begin{equation*}
\xymatrix{
\alpha:\calE_-|_{X\setminus x} \ar[r]^-\sim & \calE_+|_{X\setminus x} 
 }
\end{equation*}
of relative position $\l$ at $x$,  so that we have an equality
\begin{equation}\label{covector equality}
\xymatrix{
dp_-^\l(\phi_-) = dp_+^\l(\phi_+) + d\pi(\theta) \in T^*_{(x, \calE_-, \calE_+, \alpha)}\Hecke^\l.
 }
\end{equation}
Then to prove the Claim, we must show: 
if $\phi_- \in \calN_G(X)$, then $\phi_+\subset \calN_G(X)$ and $\theta = 0$.

Set $\Hecke^\l|_x\subset \Hecke^\l$ to be the fiber, and $j:X\setminus x\incl X$ the open inclusion.

Note that $\alpha$ induces an isomorphism of quasicoherent sheaves
\begin{equation*}
\xymatrix{
j_*j^*\frg^*_{\calE_-}  \simeq  j_*j^*\frg^*_{\calE_+}.
}
\end{equation*}
Consider the coherent subsheaf
 \begin{equation*}
\xymatrix{
\frg^*_{\calE_- \vee \calE_+}   :=\frg^*_{\calE_-} +  \frg^*_{\calE_+} \subset  j_*j^*\frg^*_{\calE_-}\simeq  j_*j^*\frg^*_{\calE_+}.
 }
\end{equation*}

Then the fiber  of the cotangent complex of $\Hecke^\l|_x$ 
is given by the complex of derived sections
\begin{equation*}
\xymatrix{
\TT^*_{(\calE_-, \calE_+, \alpha)} \Hecke^\l|_x  \simeq \Gamma(X, \frg^*_{\calE_- \vee \calE_+}  \otimes\om_{X}),
}
\end{equation*}
and that of its underlying classical cotangent bundle 
by the space of sections
\begin{equation*}
\xymatrix{
T^*_{(\calE_-, \calE_+, \alpha)} \Hecke^\l|_x  \simeq H^0(X, \frg^*_{\calE_- \vee \calE_+}  \otimes\om_{X}).
}
\end{equation*}

Thus we have a short exact sequence
\begin{equation*}
\xymatrix{
T^*_x X \ar[r] & T^*_{(x, \calE_-, \calE_+, \alpha)} \Hecke^\l  \ar[r] &
 H^0(X, \frg^*_{\calE_- \vee \calE_+}  \otimes\om_{X}).
}
\end{equation*}

The pullback of covectors
\begin{equation*}
\xymatrix{
 H^0(X,\frg^*_{\calE_{-}}\otimes\om_{X})
 \simeq T^{*}_{\calE_{-}}\Bun
\ar[r]^-{ dp_-^\l} & 
T^*_{(x, \calE_-, \calE_+, \alpha)} \Hecke^\l
\ar[r] &
 H^0(X, \frg^*_{\calE_- \vee \calE_+}  \otimes\om_{X})
 }
\end{equation*} 
is the inclusion induced by the inclusion
\begin{equation*}
\xymatrix{
\frg^*_{\calE_-} \ar@{^(->}[r] & \frg^*_{\calE_- \vee \calE_+}.
 }
\end{equation*} 
Likewise,  the pullback of covectors
\begin{equation*}
\xymatrix{
 H^0(X,\frg^*_{\calE_{+}}\otimes\om_{X})
 \simeq T^{*}_{\calE_{+}}\Bun
\ar[r]^-{ dp_+^\l} & 
T^*_{(x, \calE_-, \calE_+, \alpha)} \Hecke^\l
\ar[r] &
 H^0(X, \frg^*_{\calE_- \vee \calE_+}  \otimes\om_{X})
 }
\end{equation*} 
is  the inclusion induced by the inclusion
\begin{equation*}
\xymatrix{
\frg^*_{\calE_+} \ar@{^(->}[r] & \frg^*_{\calE_- \vee \calE_+}. 
 }
\end{equation*} 
%Note as well  that the pullbacks of covectors from $T_x^*X$  are killed after restriction to 
% $T^*\Hecke^\l|_x$.

Therefore the equality~\eqref{covector equality} implies, after passing to  
\begin{equation*}
\xymatrix{
T^*_{(\calE_-, \calE_+, \alpha)} \Hecke^\l|_x  \simeq H^0(X, \frg^*_{\calE_- \vee \calE_+}  \otimes\om_{X}),
}
\end{equation*}
 that we have the equality
\begin{equation*}
\xymatrix{
\phi_{-}|_{X\setminus x}=\phi_{+}|_{X\setminus x}.
}
\end{equation*} 
In particular, if $\phi_- \in \calN_G(X)$, then $\phi_+\in \calN_G(X)$.

Thus it remains to  show  if $\phi_- \in \calN_G(X)$, then $\theta = 0$.

Consider the coherent subsheaf
 \begin{equation*}
\xymatrix{
\frg^*_{\calE_- \wedge \calE_+}   :=\frg^*_{\calE_-} \cap  \frg^*_{\calE_+} \subset  
\frg^*_{\calE_- \vee \calE_+} . 
 }
\end{equation*}
Let $\phi$ denote the common value of $\phi_-, \phi_+$ upon passing to   
$H^0(X, \frg^*_{\calE_- \vee \calE_+}  \otimes\om_{X})$,
and note that it
lies in the subspace
\begin{equation*}
\xymatrix{
H^0(X, \frg^*_{\calE_- \wedge \calE_+}  \otimes\om_{X})\subset H^0(X, \frg^*_{\calE_- \vee \calE_+}  \otimes\om_{X}).
}
\end{equation*}

Since  the equality~\eqref{covector equality} can be rewritten as
 \begin{equation*}
\xymatrix{
d\pi(\theta)=dp_-^\l(\phi_{-})- dp_+^\l(\phi_+),
 }
\end{equation*}
it suffices to 
show: 
\begin{equation}\label{phi pm}
\mbox{When $\phi\in
H^0(X, \frg^*_{\calE_- \wedge \calE_+}  \otimes\om_{X})$ is nilpotent, we have 
$dp_-^\l(\phi_{-})= dp_+^\l(\phi_+)$.}
\end{equation}

%%%%%% red starts %%%%%%%%
%{\color{red}

We will deduce this by passing to a Grassmannian uniformization. Let us denote by 
 \begin{equation*}
\xymatrix{
(\Gr_G {\times} \Gr_G)_{\l}
\subset \Gr_G {\times} \Gr_G
 }
\end{equation*}
the subspace of the product comprising lattices in relative position $\lambda$. Similar to the construction of the map $u$ in \S\ref{sss:unif}, we have the uniformization map
 \begin{equation*}
\xymatrix{
r:(\Gr_G {\times} \Gr_G)_{\l} \times \Coord^0(X)
\ar[r] & (\Gr_G {\times} \Gr_G)_{\l} \twtimes{\Aut^{0}(\cO)} \Coord^0(X) \ar[r] & \Hecke^\l.
 }
\end{equation*}
To show \eqref{phi pm}, we will show that
\begin{equation}\label{r phi pm}
\mbox{When $\phi\in
H^0(X, \frg^*_{\calE_- \wedge \calE_+}  \otimes\om_{X})$ is nilpotent, we have 
$dr(dp_-^\l(\phi_{-}))= dr(dp_+^\l(\phi_+))$.}
\end{equation}
By the discussion in \S\ref{sss:du}, the uniformization map $u$, and hence $r$, are submersions in the sense that they induce injective maps on cotangent spaces.  However, $u$ and $r$ are guaranteed to be surjective only when $G$ is semisimple. Therefore when $G$ is semisimple,  \eqref{r phi pm} implies \eqref{phi pm}. The case of a  reductive $G$ can be reduced  to the semisimple case as follows. Let $G^{\ad}$ be the adjoint quotient of $G$, with maximal torus $T^{\ad}=T/Z(G)$, and let $\ov\l\in \L^{+}_{T^{\ad}}$ be the image of $\l$. Then we have the Hecke correspondence
\begin{equation*}
\xymatrix{
\Bun_{G^{\ad}}(X)  & \ar[l]_-{p^{\ad,\l}_-} \Hecke^{\ov\l}_{G^{\ad}}(X) \ar[r]^-{p^{\ad,\l}_+\times \pi} &   \Bun_{G^{\ad}}(X) \times X  \\
}
\end{equation*}
Since $\phi_{\pm}$ is nilpotent, it lies in $H^{0}(X, \frg^{\ad,*}_{\cE_{\pm}}\otimes\om_{X})\subset H^{0}(X, \frg^{*}_{\cE_{\pm}}\otimes\om_{X})$. For $\nu: \Hecke^{\l}=\Hecke^{\l}_{G}(X)\to \Hecke^{\ov\l}_{G^{\ad}}(X)$ the natural map, we have
\begin{equation*}
\xymatrix{dp_-^\l(\phi_{-})= d\nu\circ dp^{\ad,\l}_{-}(\phi_{-}), & dp_+^\l(\phi_{+})= d\nu\circ dp^{\ad,\l}_{+}(\phi_{+}).}
\end{equation*}
Therefore it suffices to treat the case $G=G^{\ad}$. 
%}
%%%%%%% red ends %%%%%%

From now on we assume $G$ is semisimple, in which case it suffices to show \eqref{r phi pm}. Choose $g_{-}(t), g_{+}(t)\in G(\cK), (x,t_{x})\in \Coord^{0}(X)$ such that 
\begin{equation*}
\xymatrix{           r(g_{-}(t)G(\cO), g_{+}(t)G(\cO), x,t_{x})=(x, \cE_{-}, \cE_{+}, \a).
}
\end{equation*}
Using \eqref{eq:cot coord}, the cotangent space of $(\Gr_G {\times} \Gr_G)_{\l} \times \Coord^0(X)$ at $(g_{-}(t)G(\cO), g_{+}(t)G(\cO), x,t_{x})$ can be naturally identified with
\begin{equation*}
\xymatrix{   (\Ad_{g_{-}(t)}\frg^{*}(\cO)dt+\Ad_{g_{+}(t)}\frg^{*}(\cO)dt )\oplus \Der(\cO)^{*}.      \\
}
\end{equation*}
By Proposition ~\ref{p:diff Gr BunG}, under the above decomposition, we have
\begin{equation*}
\xymatrix{    dr(dp^{\l}_{-}(\phi_{-}))=(\phi_{-}|_{D^{\times}}, \mu_{g_{-}(t)G(\cO)}(\phi_{-}|_{D^{\times}})) ,      \\
}
\end{equation*}
\begin{equation*}
\xymatrix{    dr(dp^{\l}_{+}(\phi_{+}))=(\phi_{+}|_{D^{\times}}, \mu_{g_{+}(t)G(\cO)}(\phi_{+}|_{D^{\times}})) .      \\
}
\end{equation*}
Let $\psi=\phi_{+}|_{D^{\times}}=\phi_{-}|_{D^{\times}}\in\Ad_{g_{-}(t)}\frg^{*}(\cO)dt\cap \Ad_{g_{+}(t)}\frg^{*}(\cO)dt$. Now it suffices to show if $\psi$ is nilpotent, then we have the equality
\begin{equation*}
\xymatrix{
\mu_{g_{-}(t)G(\cO)}(\psi)=\mu_{g_{+}(t)G(\cO)}(\psi)\in \Der(\cO)^{*}\simeq (\cK/\cO)dt.
}
\end{equation*}

Write $g_{+}(t)=g_{-}(t)h(t)$ (where $h(t)$ is well-defined in the double coset $G(\cO)\bs G(\cK)/G(\cO)$), so
\begin{eqnarray*}
\mu_{g_{+}(t)G(\cO)}(\psi)&=&\jiao{g'_{+}(t)g_{+}(t)^{-1}, \psi}\\
&=&\jiao{g'_{-}(t)g_{-}(t)^{-1}+\Ad_{g_{-}(t)}(h'(t)h(t)^{-1}), \psi}\\
&=&\mu_{g_{-}(t)G(\cO)}(\psi)+\jiao{h'(t)h(t)^{-1},\Ad_{g_{-}(t)^{-1}}(\psi)}\in  \Der(\cO)^{*}\simeq(\cK/\cO)dt
\end{eqnarray*}

Write $\Ad_{g_{-}(t)^{-1}}(\psi)=\eta dt$, so
\begin{equation}\label{range eta}
\xymatrix{\eta\in \frg^{*}(\cO)\cap \Ad_{h(t)}\frg^{*}(\cO).
}
\end{equation}

It remains to show that
\begin{equation}\label{h eta}
\xymatrix{
\jiao{h'(t)h(t)^{-1},\eta}\in \cO 
}
\end{equation}
for nilpotent $\eta$ satisfying \eqref{range eta}. This is the content of Lemma~\ref{lem: null pairing} immediately below, which will complete the proof of the theorem for $G$ semisimple. The general case then follows by our previous reduction.
%In general, for $G$ reductive, any $G$-bundle admits a generic trivialization, and using the factorization of the Beilinson-Drinfeld Grassmannian, we can reduce to the above argument for $G$ semisimple.
\end{proof}

\begin{lemma} \label{lem: null pairing}
Let $h(t)\in G(\cK)$ and $\eta\in \frg^{*}(\cO)\cap \Ad_{h}\frg^{*}(\cO)$ be a nilpotent element, then $\jiao{h'(t)h(t)^{-1}, \eta}\in \fm_{\cO}$, i.e., the inclusion \eqref{h eta} holds.
\end{lemma}
\begin{proof}
Since the inclusion \eqref{h eta} does not change if we multiply $h(t)$ on the left or on the right by an element in $G(\cO)$, we may assume $h(t)=t^{\l}$, for some $\l\in\L_{T}$. In this case, we reduce to showing
\begin{equation*}
\xymatrix{\jiao{\l, \eta}\in \fm_{\cO}, & \forall \eta\in \frg^{*}(\cO)\cap \Ad_{t^{\l}}\frg^{*}(\cO).
}
\end{equation*}

Using a Killing form on $\frg$, let us identify $\frg^{*}$ with $\frg$.
Let $\ov \eta$ denote the image of $\eta$ in the quotient $ \frg(\cO)/\frg(\fm_\cO) \simeq \frg$.  The condition $\eta\in \Ad_{t^{\l}}\frg(\cO)$ implies that $\ov \eta\in \frp^{-}_{\l}$, where $\frp^{-}_{\l}$ is the parabolic of $\frg$ containing $\frt=\Lie T$ with roots $\{\a|\jiao{\a,\l}\le0\}$. The Levi factor $\frl_{\l}$ of $\frp^{-}_{\l}$ has roots $\{\a|\jiao{\a,\l}=0\}$. Let $\pi:  \frp^{-}_{\l}\to \frl_{\l}$ be the projection. We have
\begin{equation*}
\xymatrix{           \ov{\jiao{\l, \eta}}=\jiao{\l, \ov\eta}=\jiao{\l, \pi(\ov\eta)}_{\frl_{\l}}
}
\end{equation*}
where the first term denotes the image of $\jiao{\l, \eta}\in\cO$ in the quotient $\cO/\fm_\cO \simeq \CC$. Since $\eta$ is nilpotent, so are $\ov\eta \in \frg$ and $\pi(\ov \eta)\in \frl_{\l}$. Since $\l$ is central in $\frl_{\l}$, we conclude that $\jiao{\l, \pi(\ov\eta)}_{\frl_{\l}}=0$, hence $\ov{\jiao{\l, \eta}}=0$ and $\jiao{\l,\eta}\in \fm_{\cO}$. This finishes the proof.
\end{proof}

The proof of Theorem \ref{th:ss} is now complete.
\end{proof}

\section{Betti spectral action}\label{s: variations}

%%%%%%%%%%%%%%%%%%%%%%%%%%%%%%%%%%%%%%%%%%%%%%%%%%%%

\subsection{Multiple Hecke modifications}
Theorem \ref{th:ss} formally extends to the case of Hecke modifications at multiple points. More precisely,  consider the iterated Hecke stack
\begin{equation*}
\xymatrix{         \Hecke_{G}(X^{n})=\Hecke_{G}(X)\times_{\Bun_{G}(X)}\Hecke_{G}(X)\times_{\Bun_{G}(X)}\cdots\times_{\Bun_{G}(X)}\Hecke_{G}(X)
}
\end{equation*}
formed using $n$ factors of $\Hecke_{G}(X)$, where the fiber product is formed using $p_{-}$ to map $\Hecke_{G}(X)$ to the copy of $\Bun_{G}(X)$ to the left of it and using $p_{+}$ to map to the copy of $\Bun_{G}(X)$ to the right of it. 
We have maps
\begin{equation}\label{eq:Hk n proj}
\xymatrix{\Bun_{G}(X) & \Hecke_{G}(X^{n}) \ar[rr]^{p_{n,+}\times\pi_{n}}\ar[l]_-{p_{n,-}} & &   \Bun_{G}(X)\times X^{n}      
}
\end{equation}
where $p_{n,-}$ is the $p_{-}$ from the first factor of $\Hecke_{G}(X)$ and $p_{n,+}$ is the $p_{+}$ from the last one.

As in the discussion in Section~\ref{sss:Sat kernel}, in the case of the iterated Hecke stack, we have a diagram
\begin{equation*}
\xymatrix{    \Hecke_{G}(X^{n})  &     \Hecke_{G}(X^{n})\times_{X^{n}}\Coord^{0}(X)^{n}\ar[r]^-{p_{n}}\ar[l]_-{q_{n}} & (G(\cO)\bs\Gr_{G})^{n}
}
\end{equation*}
with $q_{n}$ an $\Aut^{0}(\cO)^{n}$-torsor. Let
\begin{equation*}
\Sat^{0}_{G}(n):=\Sh_{c}(\Aut^{0}(\cO)^{n}\ltimes G(\cO)^{n}\bs \Gr^{n}_{G}, E)
\end{equation*}
Then for any object $\cV\in \Sat^{0}_{G}(n)$, using its $\Aut^{0}(\cO)^{n}$-equivariant structure, its pullback $p_{n}^{*}\cV$ to $\Hecke_{G}(X^{n})\times_{X^{n}}\Coord^{0}(X)^{n}$ descends to a complex $\cV'$ on $\Hecke_{G}(X^{n})$, which can be used to define a Hecke functor
\begin{equation*}
\xymatrix{        H_{n, \cV}:  \Sh(\Bun_{G}(X), E) \ar[r] & \Sh(\Bun_{G}(X)\times X^{n}, E)\\
}
\end{equation*}
\begin{equation*}
\xymatrix{          H_{n,\cV}(\cF)=(p_{n,+}\times \pi_{n})_{!}(\cV'\otimes p_{n,-}^{*}\cF).
}
\end{equation*}

\begin{theorem}\label{th:mult ss} 
For any kernel $\cV\in\Sat^{0}_{G}(n)$, the Hecke functor $H_{n,\cV}$ preserves nilpotent singular support, and for sheaves with nilpotent singular support,
it does not introduce non-zero singular codirections along the curve. In other words, for $\cF\in \Sh(\Bun_{G}(X), E)$, 
\begin{equation*}
\xymatrix{
\sing(\calF) \subset \calN_G(X) \implies 
\sing(H_{n,\cV}(\calF)) \subset \calN_G(X) \times X^{n} 
}
\end{equation*}
where $X^{n}\subset T^*X^{n}$ denotes the zero-section.
\end{theorem}
\begin{proof} 
An analogue of Lemma \ref{l:gen Sat} for $\Sat^{0}_{G}(n)$ implies that it suffices to prove the theorem for $\cV=\cV^{\un\l}$ the constant sheaf (extended by zero) on $\Gr^{\l_{1}}_{G}\times\Gr^{\l_{n}}_{G}$, for any sequence $\un\l=(\l_{1},\cdots,\l_{n})$ of dominant coweights of $G$. Consider the substack of $\Hecke_{G}(X^{n})$
\begin{equation*}
\Hecke^{\un\l}=\Hecke^{\l_{1}}_{G}(X)\times_{\Bun_{G}(X)}\Hecke^{\l_{1}}_{G}(X)\times_{\Bun_{G}(X)}\cdots\times_{\Bun_{G}(X)}\Hecke^{\l_{n}}_{G}(X).
\end{equation*}
We have a diagram
\begin{equation*}
\xymatrix{\Bun_{G}(X) & \ar[l]_-{p^{\un\l}_-} \Hecke^{\un\l}_{G}(X) \ar[r]^-{p^{\un\l}_+\times \pi_{n}} &   \Bun_{G}(X) \times X^{n} }
\end{equation*}
restricted from \eqref{eq:Hk n proj}. We have
\begin{equation*}
H_{n,\cV^{\un\l}} (\calF)\simeq H^{\un\l}(\cF):= (p_+^{\un\l}\times \pi_{n})_! (p^{\un\l}_{-})^*\calF.
\end{equation*}
We only need to show that the statement holds for $H^{\un\l}$.

We argue by induction on $n$. The case $n=1$ is Theorem \ref{th:ss}. Suppose the case $n-1$ is proved. Let $\un\l'=(\l_{1},\cdots, \l_{n-1})$ and $\cF\in \Sh_{\cN_{G}(X)}(\Bun_{G}(X), E)$. By the inductive hypothesis, $\sing(H^{\un\l'}(\cF))\subset \cN_{G}(X)\times X^{n-1}$. We have a diagram (we again use the abbreviated notation $\Bun:=\Bun_{G}(X)$, etc.)
\begin{equation*}
\xymatrix{\Bun\times X^{n-1} & \Hecke^{\l_{n}}\times X^{n-1}\ar[l]_-{p^{\l_{n}}_{-,n-1}} \ar[rr]^-{p^{\l_{n}}_{+,n-1}\times\pi} && \Bun\times X^{n-1}\times X}
\end{equation*}
given by simply taking the product of the diagram \eqref{eq:Hk lam} with $X^{n-1}$. Here $p^{\l_{n}}_{\pm,n-1}=p^{\l_{n}}_{\pm}\times\id_{X^{n-1}}: \Hecke^{\l_{n}}\times X^{n-1}\to \Bun\times X^{n-1}$. By proper base change, we have
\begin{equation*}
H^{\un\l}(\cF)=(p^{\l_{n}}_{+,n-1}\times \pi)_{!}(p^{\l_{n}}_{-,n-1})^{*}H^{\un\l'}(\cF).
\end{equation*}

The same argument as Lemma \ref{l:claim to thm} shows that it suffices to prove the naive estimate 
\begin{equation*}
\xymatrix{
(p^{\l_{n}}_{+,n-1}\times \pi)_{\nat}(d(p^{\l_{n}}_{+,n-1} \times \pi))^{-1}(dp^{\l_{n}}_{-,n-1})(p^{\l_{n}}_{-,n-1})^{-1}_{\nat} \calN_G(X) \subset \calN_G(X) \times X^{n}.
}
\end{equation*}
However, the left hand side above is exactly $\sing^{\l_{n}}\times X^{n-1}\subset T^{*}(\Bun\times X)\times T^{*}X^{n-1}$, and the above inclusion follows from the Claim preceding Lemma \ref{l:claim to thm}.
\end{proof}

\begin{remark}\label{r:RepGn action} Let $E$ be a perfect field.  Recall $\Rep(\dG_{E})$ is the abelian category of finite-dimensional $E$-representations of $\dG_{E}$. There is a natural map 
\begin{equation*}
\xymatrix{\tau'_{n}: \prod_{i=1}^{n}\Rep(\dG_{E})\simeq \prod_{i=1}^{n}\Sat^{0,\hs}_{G}\ar[r] & \Sat^{0}_{G}(n)^{\hs}}
\end{equation*}
given by external tensor product. This functor is exact in each factor. Recall Deligne's definition of the tensor product of a finite collection of $E$-linear abelian categories with finite dimensional Hom spaces and objects of finite lengths \cite[5.1, Prop. 5.13(i)]{Del-tensor}. By  \cite[Prop. 5.13(vi)]{Del-tensor}, the functor $\tau'_{n}$ canonically extends to an exact functor of the $n$-fold tensor product of the abelian category $\Rep(\dG_{E})$. By \cite[Lemme 5.21 (special case of 5.18)]{Del-tensor}, the $n$-fold tensor product of $\Rep(\dG_{E})$ is identified with $\Rep((\dG_{E})^{n})$ as a tensor category. Since $\tau'_{n}$ is tensor in each factor, the resulting functor
\begin{equation*}
\xymatrix{\tau_{n}: \Rep((\dG_{E})^{n})\ar[r] & \Sat^{0}_{G}(n)^{\hs}}
\end{equation*}
is also a tensor functor. Therefore, for each $V\in \Rep((\dG_{E})^{n})$, we have a Hecke functor 
\begin{equation*}
\xymatrix{H_{n,V}:  \Sh(\Bun_{G}(X), E) \ar[r] & \Sh(\Bun_{G}(X)\times X^{n}, E)}
\end{equation*}
defined as $H_{n,\cV}$ for $\cV=\tau_{n}(V)$.
\end{remark}

%%%%%%%%%%%%%%%%%%%%%%%%%%%%%%%%%%%%%%%%%%%%%%%%%%%%

\subsection{Level structure}

We state here a generalization of Theorem~\ref{th:ss} incorporating level structure.

Recall for a point $x\in X$, 
we write $\cO_{x}$ for the completed local ring at $x$, with maximal ideal $\fm_{x}$, and
fraction field $\cK_{x}$.

%Fix $S\subset X$ a finite subset, with complement $U = X\setminus S$,
A level structure at $x\in X$ is by definition a subgroup scheme $\KK_{x}\subset G(\cK_{x})$ that contains a congruence subgroup $G(\cO_{x})_{N}=\ker(G(\cO_{x})\to G(\cO_{x}/\fm_{x}^{N}))$, for some $N$, as a normal subgroup, and is contained in the normalizer of a maximal parahoric subgroup of $G(\cK_{x})$. 

\begin{ex}
A favorite example is $\KK_{x} \subset G(\cK_x)$ an Iwahori, or more generally a parahoric subgroup.
\end{ex}

%%%%%% red starts %%%%%%%
%{\color{red}

Let $S\subset X$ be a finite subset, and 
set $U=X\setminus S$.
 Let $\LL_{S}=(\KK_{x})_{x\in S}$ denote the choice of a level structure for each $x\in S$. 
%that is the usual unramified level structure for the constant group-scheme $G$ along an open dense $U\subset X$. 
If each $K_{x}$ is a congruence subgroup $G(\cO_{x})_{N_{x}}$, the moduli stack $\Bun_{G}(\LL_{S})$ of  $G$-bundles with $\LL_{S}$-level structures classifies $(\cE, \tau)$ where $\cE$ is a $G$-bundle over $X$ and $\tau_{x}$ is a trivialization of $\cE$ along the divisor $\sum_{x\in S}N_{x}\cdot x$. For general level structures $\LL_{S}=(\KK_{x})_{x\in S}$, the moduli stack $\Bun(\LL_S)$ of $G$-bundles with $\LL_S$-level structures is defined as follows. For each $x\in S$, pick a congruence subgroup $G(\cO_{x})_{N_{x}}$ which is normal in $\KK_{x}$. Let $\LL^{\sharp}_{S}=\{G(\cO_{x})_{N_{x}}\}_{x\in S}$ and define $\Bun_{G}(\LL_{S})$ to be the quotient of $\Bun_{G}(\LL^{\sharp}_{S})$ by $\prod_{x\in S}\KK_{x}/G(\cO_{x})_{N_{x}}$. It is easy to see that $\Bun_{G}(\LL_{S})$ is independent of the choice of $\{N_{x}\}_{x\in S}$.
%}
%%%%%%% red ends %%%%%%%%

Let $T^*\Bun(\LL_S)$ denote the classical cotangent bundle of $\Bun(\LL_S)$. Its fiber at $\calE\in \Bun(\LL_S)$ is the space
of Higgs fields given by the short exact sequence
%and maps by restriction to the space of  generic Higgs fields
\begin{equation*}
\xymatrix{
T_\cE^*\Bun(\LL_S)\ar@{^(->}[r] & H^0(U, \frg^*_{\calE|_U} \otimes\om_{X})
\ar[r] & \oplus_{x\in S} \Lie(\KK_x)^* dt_x
}
\end{equation*}
where $t_x\in \fm_x$ denotes a coordinate.

Thus it makes sense to say whether a point of
$T^*\Bun(\LL_S)$ is nilpotent by asking for its generic values to be nilpotent.
We will write $\cN(\LL_S)\subset T^*\Bun(\LL_S)$ for this global nilpotent cone.

Let $ \Sh(\Bun(\LL_S), E)$ denote the dg category of all complexes on $\Bun(\LL_S)$.
%Let $  \Sh_{\calN_G(X)}(\Bun(\LL_S))$  denote the
% full subcategory 
%of complexes with singular support in the global nilpotent cone 
%\begin{equation}
%\xymatrix{
%\sing(\calF) \subset \calN_G(X)
%}
%\end{equation}

Using the Hecke correspondence over $U$
\begin{equation*}
\xymatrix{
\Bun(\LL_{S})  & \ar[l]_-{p^{U}_-} \Hecke(\LL_{S})_{U} \ar[r]^-{p^{U}_+\times \pi_U} &   \Bun(\LL_{S}) \times U  \\
 }
\end{equation*}
we can define Hecke functors indexed by $\cV\in\Sat^{0}_{G}$
\begin{equation*}
\xymatrix{
H^{U}_{\cV}:\Sh(\Bun(\LL_{S}),E) \ar[r] & \Sh(\Bun(\LL_{S}) \times U, E) 
&
H^{U}_{\cV}(\calF)\simeq  (p^{U}_+\times \pi_U)_! (\cV'_{U}\otimes (p^{U}_{-})^{*}\calF )
}
\end{equation*}
where $\cV'_{U}$ is the spread-out of $\cV$ to $\Hecke^\l(\LL_{S})_{U}$, defined using a similar procedure as in Section~\ref{sss:Sat kernel}.

More generally, for any positive integer $n$, we have Hecke modifications at $n$ points indexed by $\cV\in\Sat^{0}_{G}(n)$
\begin{equation*}
\xymatrix{
H^{U}_{n,\cV}:\Sh(\Bun(\LL_{S}),E) \ar[r] & \Sh(\Bun(\LL_{S}) \times U^{n}, E). 
}
\end{equation*}

The following generalization of Theorem \ref{th:mult ss} can be deduced by a verbatim repeat of its proof.

\begin{theorem}\label{th:ss with ls} 
For any $\cV\in \Sat^{0}_{G}(n)$, the Hecke functor $H^{U}_{n,\cV}$ preserves nilpotent singular support, and for sheaves with nilpotent singular support, it does not introduce non-zero singular codirections along the curve. In other words, for $\cF\in \Sh(\Bun(\LL_{S}), E)$, 
\begin{equation*}
\xymatrix{
\sing(\calF) \subset \calN(\LL_S)\implies 
\sing(H^{U}_{n,\cV}(\calF)) \subset \calN(\LL_S) \times U^{n}.
}
\end{equation*}
where $U^{n}\subset T^*U^{n}$ denotes the zero-section.
\end{theorem}

%%%%%%%%%%%%%%%%%%%%%%%%%%%%%%%%%%%%%%%%%%%%%%%%%%%%

\subsection{Spectral action}\label{ss: action}

We record here our main application of Theorem~\ref{th:ss} and its generalization Theorem~\ref{th:ss with ls}.

In this subsection, let $E$ be a field of characteristic zero. Let $\dG_{E}$ be the base change of $\dG$ to $E$. We shall use the notation set out in Section ~\ref{ss:spec notation}. Recall the geometric Satake equivalence~\cite[(1.1)]{MV}
 \begin{equation*}
\xymatrix{
\Sat^\heartsuit_G \simeq \Rep(\dG_{E}).
}
\end{equation*}

\subsubsection{Unramified case}

Recall by Theorem~\ref{th:ss},
the Hecke action on $ \Sh (\Bun_G(X), E)$, at any point $x\in X$, preserves the full subcategory
$ \Sh_{\calN_G(X)} (\Bun_G(X), E)$. Restricting to kernels in $\Sat^\heartsuit_G$ provides a  $\Rep(\dG_{E})$-action 
on $ \Sh_{\calN_G(X)} (\Bun_G(X), E)$.

%%%%%%%%%%%%%%%%%%
%{\color{blue}  %beginning of blue color
%%%%%%%%%%%%%%%%%%

\begin{prop}\label{p: chiral}
Let $S$ be a topological space equipped with a deformation retraction to a point $s_0\in S$.
Given any map $f:S\to X$,  the
$\Rep(\dG_{E})$-action on $ \Sh_{\calN_G(X)} (\Bun_G(X), E)$ at any point $f(s)\in X$ is canonically isomorphic to that at  $f(s_0)\in X$. More generally, 
 the
$\Rep(\dG_{E})^{\otimes n}$-action on $ \Sh_{\calN_G(X)} (\Bun_G(X), E)$ at any collection of points $f(s_1), \ldots, f(s_n)\in X$ is canonically isomorphic to the tensor product $\Rep(\dG_{E})^{\otimes n} \to \Rep(\dG_{E})$ followed by the $\Rep(\dG_{E})$-action at  $f(s_0)\in X$.

\end{prop}

\begin{proof}
We will give details for the first assertion  of the proposition for $S = \RR, s_0 = 0$; the general case and more general assertion for multiple points
are similar (using Theorem~\ref{th:mult ss} in place of Theorem~\ref{th:ss}).

Thus for any path $\gamma:\RR\to X$,  we seek to show the 
$\Rep(\dG_{E})$-action on $ \Sh_{\calN_G(X)} (\Bun_G(X), E)$ at any point $\gamma(t)\in X $  is canonically isomorphic to that at $\gamma(0)\in X $. 

For any object
$\cF\in \Sh_{\calN_G(X)} (\Bun_G(X), E),$ and kernel
$\cV\in \Sat^\heartsuit_G$,
we have by Theorem~\ref{th:ss}
\begin{equation*}
\xymatrix{
H_\cV(\cF) \in  \Sh_{\calN_G(X) \times X} (\Bun_G(X)\times X, E).
}
\end{equation*}
Restricting along $\gamma$ we get
\begin{equation*}
\xymatrix{(\id\times\gamma)^{*}H_\cV(\cF) \in  \Sh_{\cN_G(X) \times \RR} (\Bun_{G}(X)\times \RR, E)}
\end{equation*}
whose further restriction to each $\Bun_{G}(X)\times\{t\}$ is the Hecke modification of $\cF$ at $\gamma(t)$ by the kernel $\cV$. Thus it suffices to show that for any $t\in \RR$, the restriction functor
\begin{equation}\label{eq:res t}
\xymatrix{
\res_{t}: \Sh_{\cN_{G}(X)\times\RR}(\Bun_{G}(X)\times\RR, E)\ar[r] & \Sh_{\cN_{G}(X)\times\{t\}}(\Bun_{G}(X) \times \{t\}, E)=\Sh_{\cN_{G}(X)}(\Bun_{G}(X), E)}
\end{equation} 
is an equivalence, and that for $t\in \RR$, the functor $\res^{-1}_{0}\circ\res_{t}$ is canonically isomorphic to the identity functor of $\Sh_{\cN_{G}(X)}(\Bun_{G}(X), E)$.

%and then restrict to obtain an object
% \begin{equation*}
%(u\times \gamma)^*H_\cV(\cF) \in  \Sh_{\cN_U \times \RR} (U\times \RR, E).
%\end{equation*}Namely, 
%and restricting along the path provides an object
% \begin{equation*}
%\gamma^*H_\cK(\cF) \in  \Sh_{\calN_G(X) \times [0,1]} (\Bun_G(X)\times [0,1], E)
%\end{equation*}
As discussed in Section~\ref{ss: sing stacks},
for any smooth scheme with smooth map $u:U \to \Bun_G(X)$, 
with induced correspondence
\begin{equation*}
\xymatrix{
T^*U & \ar[l]_-{du} T^*\Bun_G(X) \times_{\Bun_G(X)} U \ar[r]^-{u_{\nat}} & T^*\Bun_G(X)
}
\end{equation*}
we have the induced nilpotent Lagrangian 
\begin{equation*}
\xymatrix{
\cN_U = du(u_{\nat}^{-1}(\cN_G(X)) \subset T^*U.
}
\end{equation*}
%Given an object
%$\cF\in \Sh_{\calN_G(X)} (\Bun_G(X), E),$ and kernel
%$\cV\in \Sat^\heartsuit_G$,
%form the Hecke modification 
%\begin{equation*}
%\xymatrix{
%H_\cV(\cF) \in  \Sh_{\calN_G(X) \times X} (\Bun_G(X)\times X, E)
%&}
%\end{equation*}
%and then restrict to obtain an object
% \begin{equation*}
%(u\times \gamma)^*H_\cV(\cF) \in  \Sh_{\cN_U \times \RR} (U\times \RR, E).
%\end{equation*}
By \cite[Cor. 8.3.22]{KS}, we may choose a $\mu$-stratification $\cS = \{U_\alpha\}_{\alpha\in A}$ of $U$ such that $\cN_U\subset T^*_\cS U = \cup_{\alpha \in A} T^*_{U_\alpha} U$. Then $(u\times \gamma)^*H_\cV(\cF)$ is locally constant along the stratification
 $\cS\times \RR = \{U_\alpha \times \RR\}_{\alpha\in A}$  of $U\times \RR$ by \cite[Prop. 8.4.1]{KS}. Thus for any $t\in \RR$, the restriction functor  
\begin{equation*}
\xymatrix{\res_{U,\cS, t}:\Sh_{\cS\times\RR}(U\times\RR, E)\ar[r] & \Sh_{\cS\times\{t\}}(U \times \{t\}, E)
}
\end{equation*}
is an equivalence such that $\res^{-1}_{U,\cS,0}\circ \res_{U,\cS,t}$ is canonically isomorphic to the identity functor of $\Sh_{\cS}(U, E)$. Restricting to the full subcategories by imposing nilpotent singular support conditions, the restriction functor
\begin{equation*}
\xymatrix{
\res_{U, t}: \Sh_{\cN_{U}\times\RR}(U\times\RR, E)\ar[r] & \Sh_{\cN_{U}\times\{t\}}(U \times \{t\}, E)}
\end{equation*}
is also an equivalence such that $\res^{-1}_{U,0}\circ \res_{U,t}$ is canonically isomorphic to the identity. Since the restriction functors $\res_{U,t}$ are functorial in the map $u:U \to \Bun_G(X)$,  we conclude that $\res_{t}$ in \eqref{eq:res t} is an equivalence as desired, and that $\res^{-1}_{0}\circ\res_{t}$ is canonically isomorphic to the identity.  
%Thus indeed 
%the 
%$\Rep(\dG_{E})$-action on $ \Sh_{\calN_G(X)} (\Bun_G(X), E)$ at any point $\gamma(t)\in X $  is canonically isomorphic to that at  any other  $\gamma(t')\in X$.
\end{proof}

Recall that $\Rep(\dG_{E})\simeq \Perf(B\dG_{E})^\heartsuit $.
Let $\pi:\dG_E/\dG_E\to B\dG_E$ denote the natural projection from the adjoint quotient, and $e:B\dG_E\to \dG_E/\dG_E$ the identity section. Since $\pi$ is affine, we have the  equivalence
\begin{equation}\label{eq: ascent}
\xymatrix{
\pi_*:\QCoh(\dG_E/\dG_{E}) \ar[r]^-\sim & \Mod_{\cO_G}\QCoh(B\dG_{E})
}
\end{equation} 
 where  $\Mod_{\cO_G}\QCoh(B\dG_{E})$ denotes module objects for the algebra object
$\cO_G = \pi_*\cO_{\dG_{E}/\dG_E} = \oplus_{\lambda\in \Lambda^+} V_\lambda^\vee \otimes V_\lambda$ of functions on the group.

The following is a well-known consequence of the structure produced in Proposition~\ref{p: chiral}. 
We will give a  proof focusing on the key structures; see Remark~\ref{rem: scientific} immediately after for a more scientific approach.

\begin{prop}\label{p: tensor auts}
Given a loop $\gamma:S^1\to X$ based at a point $x_0\in X$, the $\Perf(B\dG_{E})$-action on $ \Sh_{\calN_G(X)} (\Bun_G(X), E)$ at the point $x_0\in X$
canonically extends along $\pi^*:\Perf(B\dG_{E}) \to \Perf(\dG_{E}/\dG_E)$
to a $\Perf(\dG_{E}/\dG_E)$-action. 

Moreover, given an extension of $\gamma$ to a disk $\overline \gamma:D^2\to X$, the 
 $\Perf(\dG_{E}/\dG_E)$-action canonically factors through $e^*: \Perf(\dG_{E}/\dG_E)\to \Perf(B\dG_{E})$
 followed by the original $\Perf(B\dG_E)$-action. 
\end{prop}

\begin{proof}
Consider the universal cover $p:\RR\to S^1 = \RR/\ZZ$ and the lift $\tilde \gamma = \gamma\circ p:\RR\to X$.
By Proposition~\ref{p: chiral},  the $\Perf(B\dG_{E})^{\otimes n}$-action on $ \Sh_{\calN_G(X)} (\Bun_G(X), E)$  at any collection of points $\tilde \gamma(t_1), \ldots, \tilde \gamma(t_n) \in X$
is canonically isomorphic to the tensor product  $\Perf(B\dG_{E})^{\otimes n}\to  \Perf(B\dG_{E})$ followed by the $\Perf(B\dG_{E})$-action
at $\tilde \gamma(0) =x_0$.  

Let us leave aside the symmetric monoidal structure of  $\Perf(B\dG_{E})$ for the moment and regard it as a plain monoidal category. 
By Proposition~\ref{p: chiral}, translation along $\RR$ by an integral amount provides a canonical monodromy automorphism  $m$ of the monoidal action:  automorphisms 
$m_\cV:H_{\cV, x_0}\to H_{\cV, x_0}$, for $\cV\in \Perf(B\dG_{E})$,  along with equivalences $m_{\cV_1\otimes \cV_2} \simeq m_{\cV_1} \otimes m_{\cV_2}$
and evident associativity coherences.

Let us use the monodromy automorphism  $m$ to show the Hecke functors $H_{\cV, x_0}$, for $\cV\in \Perf(B\dG_{E})$,
factor through $\pi^*:\Perf(B\dG_{E}) \to \Perf(\dG_{E}/\dG_E)$. For this, it suffices by \eqref{eq: ascent} to show $H_{\cV, x_0}$ carries a functorial  $\cO_G$-module structure given by an action map
$H_{\cO_G, x_0} \circ H_{\cV, x_0} \to H_{\cV, x_0}$ with identity and associativity coherences. From the decomposition 
$\cO_G = \oplus_{\lambda\in \Lambda^+} V_\lambda^\vee \otimes V_\lambda$, to construct the action map $H_{\cO_G, x_0} \circ H_{\cV, x_0} \to H_{\cV, x_0}$, it suffices
to define $H_{V_\lambda^\vee \otimes V_\lambda, x_0} \circ H_{\cV, x_0} \to H_{\cV, x_0}$, for $\lambda\in \Lambda^+$. We take this to be the transpose of
the monodromy automorphism $m_{V_\lambda \otimes \cV}$ under the duality of $V_\lambda$. It is a diagram chase to show  the 
canonical equivalences $m_{\cV_1 \otimes \cV_2} \simeq m_{\cV_1} \otimes m_{\cV_2}$,
and their higher coherences naturally extend this action map to an $\cO_G$-module structure on $H_{\cV, x_0}$. 

Now we have shown there is a natural factorization of the Hecke functor
\begin{equation}\label{eq: action map}
\xymatrix{
H_{x_0}:\Perf (B\dG_E) \ar[r]^-{\pi^*} & \QCoh(\dG_E/\dG_{E}) \ar[r] & \End( \Sh_{\calN_G(X)} (\Bun_G(X), E))
}
\end{equation} 
but not yet confirmed it is monoidal. For this, let us return to the picture that $\Perf(B\dG_{E})$ is not only monoidal but a tensor category. 
The canonical equivalences of Proposition~\ref{p: chiral} enhance  the monodromy automorphism $m$  to  a tensor automorphism:
there are evident coherences intertwining the symmetric structure of $\Perf(B\dG_{E})$ and the monodromy automorphisms $ m_{\cV_1 \otimes \cV_2}, m_{\cV_2 \otimes \cV_1} $. It is again a diagram chase to show that this equips
the factorization \eqref{eq: action map} with a monoidal structure.

% It also provides canonical equivalences $m_{\cV_1\otimes \cV_2} \simeq m_{\cV_1} \otimes m_{\cV_2}$

This concludes the proof of the first assertion; the second is similar and we leave to the reader.
\end{proof}

\begin{remark}\label{rem: scientific} 
Let us outline a more scientific approach to Proposition~\ref{p: tensor auts} (see also the discussion of
Remark~\ref{rem: chiral homol}) whose more detailed explanation would take us too far a field. 
Set $A= \Perf(B\dG_{E})$, $\cC = \Sh_{\calN_G(X)} (\Bun_G(X), E)$,
so that $A$ is a symmetric monoidal or $E_\infty$-category,
and the Hecke action at $x_0\in X$ gives a monoidal or $E_1$-functor $H:A\to \End(\cC)$. 

Let $\cL(\End(\cC))$ be the inertia or loop category of objects $\varphi\in \End(\cC)$
 equipped with an automorphism $\gamma\in \Aut(\varphi)$.  Since $\End(\cC)$ is an $E_1$-category, $\cL(\End(\cC))$ is an $E_2$-category, and the projection $\cL(\End(\cC))\to \End(\cC) $ is an $E_1$-functor.
  By a monoidal automorphism, let us  mean an $E_1$-lift $H^{(1)}:A\to \cL(\End(\cC))$; by a tensor automorphism, let us mean an $E_2$-lift $H^{(2)}:A\to \cL(\End(\cC))$. 

In our geometric situation, given a loop $\gamma:S^1\to X$ based at $x_0\in X$, Proposition~\ref{p: chiral}  provides a natural  tensor automorphism $H^{(2)}:A\to \cL(\End(\cC))$.
By adjunction,  this in turn provides an $E_1$-functor $A \otimes S^1\to \End(\cC)$,
where $A \otimes S^1$ denotes the Hochschild $E_1$-category of the $E_2$-category $A$. Furthermore, we have the explicit calculation $  \Perf(B\dG_E) \otimes S^1 \simeq \Perf(\dG_E/\dG_{E})$  (see for example~\cite{BFN}).

Similar considerations  apply equally well to the second assertion of  Proposition~\ref{p: tensor auts}.
 \end{remark}

Now let $\Loc_{\dG}(X)$ be the Betti derived stack of topological $\dG$-local systems on $X$.
Thus for a choice of a base-point $x_0\in X$, we have the monodromy isomorphism   
\begin{equation*}
\xymatrix{
\Loc_{\dG}(X) \simeq \Hom(\pi_1(X, x_0), \dG)/\dG
}
\end{equation*}

More concretely, regarding $X$ as a topological surface, fix a standard basis 
$\alpha_1,  \ldots, \alpha_g$,
$\beta_1,  \ldots, \beta_g$
of loops  based at $x_0$ so that we have $\prod_{i=1}^g [\alpha_i, \beta_i] = 1 \in \pi_1(X, x_0)$.
Given a $\dG$-local system, with a trivialization of its fiber at $x_0$, taking monodromy around these loops gives
a  collection of elements $g(\alpha_i), g(\beta_i)\in G$, for $i = 1,\ldots, g$. 
In this way, we obtain  a Cartesian presentation
\begin{equation*}
\xymatrix{
\ar[d] \Loc_{\dG}(X)  \ar[r] & ((\dG)^{g}\times (\dG)^{g})/\dG\ar[d]^-c \\
B\dG \ar[r]^-e  & \dG/\dG
}
\end{equation*}
where the map $c$ is induced by $\prod_{i=1}^g [g(\alpha_i), g(\beta_i)]$, and the map $e$ by the inclusion of the identity.

To understand perfect complexes on $\Loc_{\dG}(X)_{E}$ via the above Cartesian square,  let us follow the proof of \cite[Proposition 4.13]{BFN}. Its hypotheses do not hold here, since neither $B\dG_{E}$ 
nor $((\dG_{E})^{g}\times (\dG_{E})^{g})/\dG_{E}$ is affine,
but in fact all that 
 its proof uses is that 
 $B\dG_{E}$ has affine diagonal, and $e$ and $c$ are affine.
 Namely,  the fact that $e$ is affine implies pushforward along it induces an equivalence
\begin{equation*}
\xymatrix{
\QCoh(B\dG_{E}) \ar[r]^-\sim & \Mod_{e_*\cO_{B\dG_{E}}}\QCoh(\dG_{E}/\dG_{E})
}
\end{equation*} 
 where  $\Mod_{e_*\cO_{B\dG_{E}}}\QCoh(\dG_{E}/\dG_{E})$ denotes module objects for the algebra object
 $e_*\cO_{B\dG_{E}} \in \QCoh(\dG_{E}/\dG_{E})$ of functions on the adjoint-orbit of the identity. By \cite[Proposition 4.1]{BFN},
 this in turn implies that
 $\QCoh(B\dG_{E})$ is self-dual, and in particular dualizable, as a  $\QCoh(\dG_{E}/\dG_{E})$-module.  Along with the fact that  $B\dG_{E}$ has affine diagonal, and $e$ and $c$ are affine, the limit calculations of the proof of \cite[Proposition 4.13]{BFN} 
 only depend on this dualizability.
 Thus
 we conclude  that pullback induces a tensor equivalence
\begin{equation*}
\xymatrix{
\QCoh(\Loc_{\dG}(X)_{E}) & \ar[l]_-\sim  \QCoh(((\dG_{E})^{g}\times (\dG_{E})^{g})/\dG_{E} ) \otimes_{\QCoh(\dG_{E}/\dG_{E})} \QCoh(B\dG_{E}) 
}
\end{equation*}
Since this equivalence preserves compact objects, we obtain a similar tensor equivalence
\begin{equation}\label{eq:tensor Perf Loc}
\xymatrix{
\Perf(\Loc_{\dG}(X)_{E}) & \ar[l]_-\sim  \Perf(((\dG_{E})^{g}\times (\dG_{E})^{g})/\dG_{E} ) \otimes_{\Perf(\dG_{E}/\dG_{E})} \Perf(B\dG_{E}) 
}
\end{equation}
by recalling the compatibility of the tensor product of small stable $\infty$-categories 
and presentable stable $\infty$-categories under taking $\Ind$ and conversely passing to compact objects.

With the preceding in hand, we can now conclude the following.

\begin{theorem}[Betti spectral action]\label{thm: action}
Let $E$ be a field of characteristic zero. There is an $E$-linear tensor action
\begin{equation*}
\xymatrix{
\Perf(\Loc_{\dG}(X)_{E}) \curvearrowright \Sh_{\calN_G(X)} (\Bun_G(X), E) }
\end{equation*}
%\begin{equation}
%\xymatrix{
%\int_X \Rep(\dG,E)\curvearrowright \Sh_{\calN_G(X)} (\Bun_G(X), E) }
%\end{equation}
such that for any point $x\in X$, its restriction via pullback along the natural evaluation
 \begin{equation*}
\xymatrix{
 \Rep(\dG_{E}) \ar[r]^-{\ev_{x}^{*}} & \Perf(\Loc_{\dG}(X)_{E})}
\end{equation*}
is isomorphic, under the Geometric Satake correspondence, to the Hecke action
of $\Sat^\heartsuit_G$ at the point $x\in X$.
\end{theorem}

\begin{proof}

By the geometric Satake correspondence, the Hecke action  by $\Sat^\heartsuit_G$ at the base-point $x_0\in X$ provides a  $\Rep(\dG_{E})$-action 
on $ \Sh_{\calN_G(X)} (\Bun_G(X), E)$.
 By the first assertion of Proposition~\ref{p: tensor auts},  the  based loops $\alpha_1,  \ldots, \alpha_g$,
$\beta_1,  \ldots, \beta_g$ provide
a lift to a $\Perf(\dG_{E}/\dG_E)^{\otimes 2g}$-action. 
  Applying the second assertion of Proposition~\ref{p: tensor auts} to a disk with boundary the composition 
  $\prod_{i=1}^g [\alpha_i, \beta_i]$
  of the based loops, we see that  the
    $\Perf(\dG_{E}/\dG_E)^{\otimes 2g}$-action factors through a $\Perf(\Loc_{\dG}(X)_{E})$-action. 
\end{proof}

%%%%%%%%%%%%%%%%%%
%} %end of blue color
%%%%%%%%%%%%%%%%%%

\subsubsection{With level structure}

Let $S\subset X$ be a finite subset, and 
set $U=X\setminus S$.
 Let $\LL_{S}=(\KK_{x})_{x\in S}$ denote the choice of a level structure for each $x\in S$.

Let $\Loc_{\dG}(U)$ be the Betti derived stack of topological $\dG$-local systems on $U$.
Thus for a choice of a base-point $u_0\in U$, we have the monodromy isomorphism   
\begin{equation*}
\xymatrix{
\Loc_{\dG}(U) \simeq \Hom(\pi_1(U, u_0), \dG)/\dG.
}
\end{equation*}
Assuming $S$ is nonempty, so that $U$ is homotopy equivalent to a bouquet of $n$ circles, 
we may choose based loops so that  the monodromy isomorphism   takes the form
\begin{equation*}
\xymatrix{
\Loc_{\dG}(U) \simeq (\dG)^n/\dG.
}
\end{equation*}

The following generalization of Theorem~\ref{thm: action} can be deduced from Theorem~\ref{th:ss with ls}
by the same argument. (In fact, assuming $S$ is nonempty, it only involves the first two steps, but not the third, since there is no equation to impose.)

\begin{theorem}[Betti spectral action with level structure]\label{thm: action with ls}
Let $E$ be a field of characteristic zero. There is an $E$-linear tensor action
\begin{equation*}
\xymatrix{
\Perf(\Loc_{\dG}(U)_{E}) \curvearrowright  \Sh_{\calN(\LL_S)} (\Bun(\LL_S), E)% \ar[r] &  \Sh_{\calN_G(X)} (\Bun_G(X))
}
\end{equation*}
such that for any point $u\in U$, its restriction via pullback along  the natural evaluation
\begin{equation*}
\xymatrix{
  \Rep(\dG_{E}) \ar[r]^-{\ev_u^*} & \Perf(\Loc_{\dG}(U)_{E})
}
\end{equation*}
is isomorphic, under the Geometric Satake correspondence, to the Hecke action
of $\Sat^\heartsuit_G$ at the point $u\in U$.
\end{theorem}

%%%%

\subsubsection{With tame ramification}
Let $S\subset X$ be a finite subset, and 
set $U=X\setminus S$.
 We consider here level structure $\LL_{S}=(\II_{x})_{x\in S}$ that is Iwahori level structure at each $x\in S$. 
 To simplify the notation, we will only  include $S$, rather than $\LL_S$, in the notation for the resulting objects. For example,  $\Bun_G(X, S)$ denotes the moduli of of $G$-bundles on $X$ with a $B$-reduction at each $s\in S$.
Let $\calN_G(X, S)\subset T^*\Bun_G(X, S)$ denote the global nilpotent cone.

Let $\Loc_{\dG}(X, S)$ denote the moduli of $\dG$-local systems on $U$ equipped near $S$
with  $B^\vee$-reductions 
with trivial induced $H^\vee$-monodromy.

Thus for a choice of a base-point $u_0\in U$, we have the monodromy isomorphism   
\begin{equation*}
\xymatrix{
\Loc_{\dG}(X, S) \simeq (\Hom(\pi_1(U, u_0), \dG)
 \times_{ (\dG)^S} (\wt \cN^\vee)^S) /\dG
}
\end{equation*}
%an induced $\dG$-equivariant tensor equivalence
%\begin{equation*}
%\xymatrix{
%\Perf(\Loc_{\dG}(X, S) ) & \ar[l]_-\sim  
%\Perf(\Hom(\pi_1(U, u_0), \dG))
% \otimes_{ \Perf(\dG)^{\otimes S}} \Perf(\wt \dG)^{\otimes S}
%}
%\end{equation*}
and an induced tensor equivalence ($E$ is a characteristic zero field)
\begin{equation*}
\xymatrix{
\Perf(\Loc_{\dG}(X, S)_{E}) & \ar[l]_-\sim  
\Perf(\Hom(\pi_1(U, u_0), \dG_{E})/\dG_{E})
 \otimes_{ \Perf(\dG_{E}/\dG_{E})^{\otimes S}} \Perf(\wt \cN^\vee_{E}/\dG_{E})^{\otimes S}.
}
\end{equation*}
%Note, assuming $S$ is nonempty, so that $U$ is homotopy equivalent to a bouquet of $n$ circles, 
%we also have $\Hom(\pi_1(U, u_0), \dG) \simeq (\dG)^n$.

By Bezrukavnikov's tame local Langlands correspondence~\cite[Theorem 1(4)]{B}, at each $s\in S$, we have a monoidal equivalence
\begin{equation*}
\xymatrix{      \Coh(\St_{\dG,E}/\dG_{E}) \simeq \Sh_{c}(\II_{s}\bs G((t_{s}))/\II_{s}, E)
}
\end{equation*}
where $\St_{\dG}=\tdN\times_{\dG}\tdN$ is the derived Steinberg variety over $ E$, and $\St_{\dG,E}$ is its base change to $E$. In particular, via the diagonal embedding $\D:\tdN\incl \St_{\dG}$, we have a monoidal functor
\begin{equation}\label{eq: Waki}
\xymatrix{       \Perf(\tdN_{E}/\dG_{E}) \ar[r]^-{\D_{*}}  & \Coh(\St_{\dG,E}/\dG_{E}) \simeq \Sh_{c}(\II_{s}\bs G((t_{s}))/\II_{s}, E).
}
\end{equation}
Since the Iwahori-Hecke category $\Sh_{c}(\II_{s}\bs G((t_{s}))/\II_{s}, E)$  acts on $\Sh(\Bun_G(X, S), E)$ by bundle modification at $s$, we get commuting actions of $\Perf(\tdN_{E}/\dG_{E})$ on $\Sh(\Bun_G(X, S),E)$ for each $s\in S$. By Bezrukavnikov's construction \cite[4.1.2]{B},
the restriction of the actions along the pullback $\Rep(\dG_{E})\simeq\Perf(B\dG_{E}) \to \Perf(\tdN_{E}/\dG_{E})$ is equivalent to the nearby cycles of the Satake action of $\Rep(\dG_{E})$ at nearby points. Moreover, the monodromy of the nearby cycles provides the lift of the restriction to   the pullback
$\Perf(\dG_{E}/ \dG_{E}) \to \Perf(\tdN_{E}/\dG_{E})$.

With the preceding in hand, the following tamely ramified version of Theorem~\ref{thm: action} can be deduced 
by the same argument.

\begin{theorem}[tamely ramified Betti spectral action]\label{thm: action tame}
Let $E$ be a field of characteristic zero.
There is  a tensor action
 \begin{equation*}
\xymatrix{
\Perf(\Loc_{\dG}(X, S)_{E}) \curvearrowright \Sh_{\calN_G(X, S)} (\Bun(X, S),E) % \ar[r] &  \Sh_{\calN_G(X)} (\Bun_G(X))
}
\end{equation*}
such that for any point $u\in U$, its restriction via pullback along  the natural evaluation
 \begin{equation*}
\xymatrix{
 \Rep(\dG_{E})\ar[r]^-{\ev^{*}_{u}} & \Perf(\Loc_{\dG}(X, S)_{E})
}
\end{equation*}
 is isomorphic, under the Geometric Satake correspondence, to the Hecke action
of $\Sat^\heartsuit_G$ at the point $u\in U$,
and for any point $s\in S$, its restriction via pullback along  the natural evaluation
 \begin{equation*}
\xymatrix{
\ev_s^*:\Perf(\tdN_{E}/\dG_{E})  \ar[r] & \Perf(\Loc_{\dG}(X, S)_{E})
}
\end{equation*}
is isomorphic to the action of $\Perf(\tdN_{E}/\dG_{E})$ at the point $s\in S$ given by \eqref{eq: Waki}.
\end{theorem}

%\begin{remark}
%The theorem also admits a natural generalization to monodromic sheaves.
%\end{remark}

%%%%%%%%%%%%%%%%%%%%%%%%%%%%%%%%%%%%%%%%%%%%%%%%%%%%
\section{Betti excursion operators and Betti Langlands parameters}
%%%%%%%%%%%%%%%%%%%%%%%%%%%%%%%%%%%%%%%%%%%%%%%%%%%%
One can use the local constancy of the Hecke functors proved in Theorem~\ref{th:ss with ls} to define a subalgebra of the center of the automorphic category, analogous to the ``excursion operators'' defined by V.~Lafforgue. Using these operators, we can  associate a Betti Langlands parameter to an indecomposable automorphic complex.   

In this subsection, $E$ is an algebraically closed field.

Let $S\subset X$ be a finite subset, and 
set $U=X\setminus S$. Set $\Gamma=\pi_{1}(U,u_{0})$, which is a finitely presented group. 
 Let $\LL_{S}=(\KK_{x})_{x\in S}$ denote the choice of a level structure for each $x\in S$.

\subsection{Betti excursion operators}\label{s:Betti ex}
In \cite{Laf}, V.~Lafforgue constructed a collection of operators on the space of cusp forms called ``excursion operators'' using moduli of iterated Shtukas. In the setting of Betti geometric Langlands, we have an analogous construction, now acting on each object $\cF\in \Sh_{\calN(\LL_S)} (\Bun(\LL_S), E)$ (or, acting on the identity functor of $\Sh_{\calN(\LL_S)} (\Bun(\LL_S), E)$).

For any $n\ge1$, consider $\cO_{n}=\cO(\dG_{E})^{\otimes n}= \cO((\dG_{E})^{n})$. Consider the action of $(\dG)^{n+1}$ on $(\dG)^{n}$ given by
\begin{equation*}
(h_{0},h_{1},\cdots, h_{n})\cdot (g_{1},g_{2},\cdots ,g_{n})=(h_{0}g_{1}h^{-1}_{1}, h_{0}g_{2}h^{-1}_{2},\cdots, h_{0}g_{n}h^{-1}_{n}).
\end{equation*}
This way $\cO_{n}$ becomes a representation of $(\dG)^{n+1}$. By Remark~\ref{r:RepGn action}, each object $V\in \Rep((\dG)^{n+1}, E)$ defines a Hecke functor $H_{n+1,V}$. Passing to ind-objects, the ind-object $\cO_{n}\in \Ind-\Rep((\dG)^{n+1}, E)$ defines a Hecke functor 
\begin{equation*}
\xymatrix{   H_{n+1,\cO_{n}}:\Sh_{\calN(\LL_S)} (\Bun(\LL_S), E)\ar[r] & \Sh_{\calN(\LL_S)} (\Bun(\LL_S)\times U^{n+1}, E)
}
\end{equation*}

If we restrict to the diagonal $\D: \dG\incl (\dG)^{n+1}$, the induced action of $\dG$ on $(\dG)^{n}$ is by simultaneous conjugation. 

Consider the tautological inclusion
\begin{equation*}
\xymatrix{      \cO((\dG_{E})^{n}/\dG_{E})   = \cO_{n}^{\D(\dG_{E})} \ar@{^{(}->}[r] & \cO_{n}
}
\end{equation*}
as $\D(\dG_{E})$-modules.  This induces a natural transformation 
\begin{equation*}
\xymatrix{c^{\sharp}: \cO((\dG_{E})^{n}/\dG_{E})\otimes \id\ar[r] & H_{n+1,\cO_{n}}|_{\D(u_{0})}
}
\end{equation*}
of endo-functors on $\Sh_{\calN(\LL_S)} (\Bun(\LL_S), E)$.
Here $H_{n+1,\cO_{n}}|_{\D(u_{0})}$ is the composition of $H_{n+1,\cO_{n}}$ with the restriction to $\Bun(\LL_S)\times\{\D(u_{0})\}$. 

Now by Theorem \ref{th:ss with ls}, the functor $H_{n+1,\cO_{n}}|_{\D(u_{0})}$ carries an action by the fundamental group $\pi_{1}(U^{n+1}, \Delta(u_{0}))\simeq \Gamma^{n+1}$. In particular, for any $\un\g=(\g_{1},\cdots, \g_{n})\in\Gamma^{n}$, we may consider the action 
\begin{equation*}
\xymatrix{(1,\un\g):        H_{n+1,\cO_{n}}|_{\D(u_{0})}\ar[r]   & H_{n+1,\cO_{n}}|_{\D(u_{0})}
}
\end{equation*}  
of $(1,\g_{1},\cdots, \g_{n})$ on $H_{n+1,\cO_{n}}|_{\D(u_{0})}$.
Finally, evaluation at the identity gives a $\D(\dG)$-invariant map $\cO_{n}\to E$, hence a natural transformation
\begin{equation*}
\xymatrix{  c^{\flat}:        H_{n+1,\cO_{n}}|_{\D(u_{0})}\ar[r] &  H_{1,E}=\id
}
\end{equation*}
The composition
\begin{equation*}
\xymatrix{\cO((\dG_{E})^{n}/\dG_{E})\otimes \id\ar[r]^-{c^{\sharp}} & H_{n+1,\cO_{n}}|_{\D(u_{0})}\ar[r]^{(1,\un\g)} & H_{n+1,\cO_{n}}|_{\D(u_{0})}\ar[r]^-{c^{\flat}} & \id}
\end{equation*}
defines an $E$-linear map
\begin{equation*}
\xymatrix{S_{\un\g}: \cO((\dG_{E})^{n}/\dG_{E})\ar[r] &  \End(\id_{\Sh_{\calN(\LL_S)} (\Bun(\LL_S), E)})=:\cZ(\Sh_{\calN(\LL_S)} (\Bun(\LL_S), E)).
}
\end{equation*}
It is easy to check that this is indeed an $E$-algebra map. 

We may call $S_{\un\g}$ the {\em Betti excursion operator} associated to the $n$-tuple $\un\g\in \Gamma^{n}$. For each object $\cF\in \Sh_{\calN(\LL_S)} (\Bun(\LL_S), E)$, we get an action of $ \cO((\dG_{E})^{n}/\dG_{E})$ on $\cF$
\begin{equation}\label{eq:Ex F}
\xymatrix{    S_{\un\g, \cF}: \cO((\dG_{E})^{n}/\dG_{E})\ar[r] & Z(\End(\cF))
}
\end{equation}
where $\End(\cF)$ is the (plain) $E$-algebra of endomorphisms of $\cF$ in the homotopy category of $\Sh_{\calN(\LL_S)} (\Bun(\LL_S), E)$, and $Z(\End(\cF))$ its (underived) center.

The assignment $\un\g\mapsto S_{\un\g}$ defines a {\em universal excursion operator}
\begin{equation*}
\xymatrix{      \Th_{n}:      \cO((\dG_{E})^{n}/\dG_{E})\ar[r]  & C(\Gamma^{n}, \cZ(\Sh_{\calN(\LL_S)} (\Bun(\LL_S), E))).
}
\end{equation*}
Here the right hand side is the ring of $\cZ(\Sh_{\calN(\LL_S)} (\Bun(\LL_S), E))$-valued functions on $\Gamma^{n}$, under pointwise ring operations.

The following proposition is an analogue of \cite[Definition-Proposition 11.3(c)(d)]{Laf}, with the same formal aspects of the proof. The local constancy of the Hecke action proved in Theorem \ref{th:mult ss} replaces the use of Drinfeld lemma in \cite{Laf}.

\begin{prop}\label{p:univ ex} The universal excursion operators $\Th_{n}$ satisfy the following properties.
\begin{enumerate}
\item For any $m,n\ge1$ and any map $\z: \{1,2,\cdots, m\}\to \{1,2,\cdots, n\}$, the following diagram is commutative
\begin{equation*}
\xymatrix{   \cO((\dG_{E})^{m}/\dG_{E}) \ar[r]^-{\Th_{m}}\ar[d]^{(-)^{\z}} &     C(\Gamma^{m}, \cZ(\Sh_{\calN(\LL_S)} (\Bun(\LL_S), E)))\ar[d]^{(-)^{\z}}    \\
\cO((\dG_{E})^{n}/\dG_{E}) \ar[r]^-{\Th_{n}} &     C(\Gamma^{n}, \cZ(\Sh_{\calN(\LL_S)} (\Bun(\LL_S), E)))}
\end{equation*}
Here the vertical maps labeled $(-)^{\z}$ are induced by the natural maps
\begin{equation*}
\xymatrix{    (\dG)^{n}\ar[r] & (\dG)^{m},       & \Gamma^{n}\ar[r] & \Gamma^{m}
}
\end{equation*}
given by $(g_{1},\cdots, g_{n})\mapsto (g_{\z(1)},\cdots, g_{\z(m)})$.

\item For any $n\ge1$, the following diagram is commutative
\begin{equation*}
\xymatrix{         \cO((\dG_{E})^{n}/\dG_{E}) \ar[r]^-{\Th_{n}}\ar[d]^{\mu_{n,n+1}} &     C(\Gamma^{n}, \cZ(\Sh_{\calN(\LL_S)} (\Bun(\LL_S), E)))\ar[d]^{\mu_{n,n+1}}    \\
\cO((\dG_{E})^{n+1}/\dG_{E}) \ar[r]^-{\Th_{n+1}} &     C(\Gamma^{n+1}, \cZ(\Sh_{\calN(\LL_S)} (\Bun(\LL_S), E)))
}
\end{equation*}
where the vertical maps labeled by $\mu_{n,n+1}$ are induced by the maps
\begin{equation*}
\xymatrix{    (\dG)^{n+1}\ar[r] & (\dG)^{n},       & \Gamma^{n+1}\ar[r] & \Gamma^{n}       
}
\end{equation*}
given by $(g_{1},\cdots, g_{n},g_{n+1})\mapsto (g_{1},\cdots, g_{n-1}, g_{n}g_{n+1})$.
\end{enumerate}
\end{prop}

\subsection{The geometric $\bR\to \bT$ map}\label{s:R to T}
Let $R=H^0(\Loc_{\dG}(U)_{E},\cO)$ be the ring of regular functions on the (non-derived) moduli stack of $\Loc_{\dG}(U)_{E}\simeq \Hom(\Gamma, \dG_{E})/\dG_{E}$. For any $\un\g=(\g_{1},\cdots, \g_{n})\in\Gamma^{n}$, we have an evaluation map 
\begin{equation*}
\xymatrix{\ev_{\un\g}: \Loc_{\dG}(U)_{E}\ar[r] & (\dG_{E})^{n}/\dG_{E}.}
\end{equation*}
Here $(\dG_{E})^{n}/\dG$ is the quotient of $(\dG_{E})^{n}$ by the diagonal conjugation action of $\dG_{E}$. Pulling back functions, we get an $E$-algebra map
\begin{equation*}
\xymatrix{\ev_{\un\g}^{*}:\cO((\dG_{E})^{n}/\dG_{E})\ar[r] & R.}
\end{equation*} 
Equivalently, we have an $E$-algebra map
\begin{equation*}
\xymatrix{\ev_{n}:\cO((\dG_{E})^{n}/\dG_{E})\ar[r] & C(\Gamma^{n},R).}
\end{equation*}

The maps $\{\ev_{n}\}$ satisfy the same properties as $\{\Th_{n}\}$ listed in Proposition \ref{p:univ ex} with $\cZ(\Sh_{\calN(\LL_S)} (\Bun(\LL_S), E))$ replaced by $R$. These constructions work for any semigroup $\Gamma$ (monoid without unit), with $R=R_{\Gamma}$ defined as the ring of regular functions on the stack $\Hom(\Gamma, \dG_{E})/\dG_{E}$ classifying conjugacy classes of semigroup maps $\Gamma\to \dG_{E}$. When $\Gamma$ is group, semigroup maps $\Gamma\to \dG_{E}$ are the same as group homomorphisms.

Now consider the category $\frR_{\Gamma}$ of $E$-algebras $R'$ equipped with algebra maps $\th_{n}: \cO((\dG_{E})^{n}/\dG_{E})\to C(\Gamma^{n}, R')$ satisfying the properties listed in Proposition \ref{p:univ ex}.

\begin{lemma} For any semigroup $\Gamma$, the category $\frR_{\Gamma}$ has an initial object $(R^{\univ}_{\Gamma}, \{\ev^{\univ}_{n}\})$.
\end{lemma}
\begin{proof} For any $n\ge1$, let $F^{+}_{n}$ be the free monoid in $n$ generators. Then $\Hom(F_{n}^{+}, \Gamma)=\Gamma^{n}$. Any map of free monoids $\ph: F_{m}^{+}\to F_{n}^{+}$ induces a map $\ph_{\Gamma}: \Gamma^{n}=\Hom(F_{n}^{+},\Gamma)\xrightarrow{(-)\circ\ph} \Hom(F_{m}^{+},\Gamma)=\Gamma^{m}$, hence a pullback map $\ph^{*}_{\Gamma,R'}: C(\Gamma^{m}, R')\to C(\Gamma^{n},R')$ for any ring $R'$. Similarly, $\ph$ induces $\ph_{\dG}: (\dG)^{n}\to (\dG)^{m}$ and $\ph^{*}_{\dG}: \cO((\dG_{E})^{m}/\dG_{E})\to \cO((\dG_{E})^{n}/\dG_{E})$. Now the properties in Proposition \ref{p:univ ex} for $(R',\{\th_{n}\})\in\frR$ are equivalent to the statement that for any map of free monoids $\ph: F_{m}^{+}\to F_{n}^{+}$, the following diagram is commutative
\begin{equation*}
\xymatrix{     \cO((\dG_{E})^{m}/\dG_{E})\ar[r]^-{\th_{m}}\ar[d]^{\ph^{*}_{\dG}} &    C(\Gamma^{n}, R')\ar[d]^{\ph^{*}_{\Gamma,R'}}   \\
\cO((\dG_{E})^{n}/\dG_{E}) \ar[r]^-{\th_{n}} & C(\Gamma^{n}, R')}
\end{equation*}

We first consider the case $\Gamma$ is finitely generated as a semigroup. Choose a finite set of generators $e_{1},\cdots, e_{n}$ for $\Gamma$. Let $\pi: F^{+}_{n}\to \Gamma$ be the map sending the $i$-th generator to $e_{i}$. Let $I\subset \cO((\dG_{E})^{n}/\dG_{E})$ be the ideal generated by all elements of the form $\ph_{\dG}^{*}(f)-\psi_{\dG}^{*}(f)$ where $\ph,\psi: F^{+}_{m}\to F^{+}_{n}$ are two maps of semigroups such that $\pi\circ\ph=\pi\circ\psi$, and $f\in \cO((\dG_{E})^{m}/\dG_{E}), m\ge1$.  We let
\begin{equation*}
\xymatrix{   R^{\univ}=\cO((\dG_{E})^{n}/\dG_{E})/I.
}
\end{equation*}
We then define $\ev^{\univ}_{m}$ for any $m\ge1$. For any $\un\g\in\Gamma^{m}$, the corresponding map $\un\g: F_{m}^{+}\to \Gamma$ factors as $F^{+}_{m}\xrightarrow{\ph}F^{+}_{n}\xrightarrow{\pi}\Gamma$, and we then define
\begin{equation*}
\xymatrix{       \ev^{\univ}_{m}(\un\g): \cO((\dG_{E})^{m}/\dG_{E})\ar[r]^-{\ph_{\dG}^{*}} & \cO((\dG_{E})^{n}/\dG_{E})\ar@{->>}[r] & R.}
\end{equation*}
The construction of $R$ implies that $\ev^{\univ}_{m}(\un\g)$ is independent of the choice of $\ph$.  It is then easy to check that $(R^{\univ}_{\Gamma}, \{\ev^{\univ}_{m}\})$ is an initial object in $\frR_{\Gamma}$.

For general $\Gamma$, we write $\Gamma$ as a filtered colimit of  finitely generated sub-semigroups $\Gamma_{\a}$.  Then $R_{\Gamma}$ is the filtered colimit of $R_{\Gamma_{\a}}$. The initial object in $\frR_{\Gamma}$ is also the filtered colimit of the initial objects in $\frR_{\Gamma_{\a}}$. Therefore $R_{\Gamma}$ is the initial object in $\frR_{\Gamma}$.
\end{proof}

For our $\Gamma=\pi_{1}(U,u_{0})$, denote by $(R^{\univ},\{\ev^{\univ}_{n}\})$ the universal object in $\frR_{\Gamma}$.
\begin{cor}\label{cor: R to T} There are canonical $E$-algebra maps
\begin{equation}\label{eq:RRT}
\xymatrix{R=\upH^{0}(\Loc_{\dG}(U)_{E}, \cO) & R^{\univ}\ar[l]_-{\om}\ar[r]^-{\sigma} & \cZ(\Sh_{\calN(\LL_S)} (\Bun(\LL_S), E)).
}
\end{equation}
\end{cor}

%%%%% red starts %%%%%%
%{\color{red}

\begin{prop}\label{p:R to T} In either of the following situations, there is a canonical ring homomorphism
\begin{equation}
\xymatrix{ \ov\sigma: R \ar[r] & \cZ(\Sh_{\calN(\LL_S)} (\Bun(\LL_S), E))
}
\end{equation}
such that $\sigma=\ov\sigma\circ \om$.
\begin{enumerate}
\item When $\Gamma$ is a free group, in which case $\om$ is an isomorphism.
\item When $\textup{char}(E)=0$.
\end{enumerate}
\end{prop}
\begin{proof}
In situation (1), if $\Gamma$ is a free group of rank $n$, then both $R^{\univ}$ and $R$ are canonically isomorphic to $\cO((\dG_{E})^{n}/\dG_{E})$. 

In situation (2), by Theorem \ref{thm: action with ls}, $\Perf(\Loc_{\dG}(U)_{E})$ acts on $\Sh_{\calN(\LL_S)} (\Bun(\LL_S), E)$, therefore the endomorphism ring of the tensor unit $\cO_{\Loc_{\dG}(U)_{E}}\in \Perf(\Loc_{\dG}(U)_{E})$, which is $R$, maps to the center of  $\Sh_{\calN(\LL_S)} (\Bun(\LL_S), E)$. This gives the map $\ov\sigma$. The factorization $\sigma=\ov\sigma\circ \om$ follows from the construction.
\end{proof}

\begin{remark}
The map $\ov\sigma$ in Proposition \ref{p:R to T} should be thought of as an analogue of the $\bR\to \bT$ map in number theory, where $\bR$ is a universal deformation ring of Galois representations, and $\bT$ is a certain Hecke algebra.
\end{remark}
%\begin{enumerate}
%\item 
%\item When $\Gamma$ is a free group (this is the case if we impose nontrivial level structures on $X$) of rank $n$, it is easy to see that $\om: R^{\univ}\to R$ is an isomorphism, both being isomorphic to $\cO((\dG_{E})^{n}/\dG_{E})$. 
%\item When $\textup{char}(E)=0$, $\om$ is always surjective.  Proposition \ref{p:ss rep} below implies that $\om$ induces an isomorphism on the reduced rings $R^{\univ,\red}\cong R^{\red}$.   But in this case  a stronger statement follows from Theorem \ref{thm: action with ls}, namely  the map $\sigma$  factors through $R$.  Indeed, since $\Perf(\Loc_{\dG}(U)_{E})$ acts on $\Sh_{\calN(\LL_S)} (\Bun(\LL_S), E)$, the center of $\Perf(\Loc_{\dG}(U)_{E})$, which is $R$, then acts on the center of  $\Sh_{\calN(\LL_S)} (\Bun(\LL_S), E)$.
%\end{enumerate}

In general we do not know whether $\om: R^{\univ}\to R$ is an isomorphism, but the next proposition shows that $\om$ always induces a bijection on the closed points of $\Spec R$ and $\Spec R^{\univ}$.

%}
%%%%%%% red ends %%%%%%%%%%

Recall from \cite[3.2.1]{Serre} that a homomorphism $\rho :\Gamma\to \dG(E)$ is called {\em completely reducible} if for any parabolic $P\subset \dG_{E}$ containing $\rho(\Gamma)$, there exists a Levi subgroup of $P$ that still contains $\rho(\Gamma)$. Let $\Hom^{cr}(\Gamma, \dG(E))\subset \Hom(\Gamma,\dG(E))$ be the subset of completely reducible homomorphisms.

\begin{prop}[Richardson, Bate-Martin-R\"ohrle, V.~Lafforgue]\label{p:ss rep} The affinization map $\Loc_{\dG}(U)_{E} \to \Spec  R$ and $\om$ induce  bijections of sets
\begin{equation*}
\xymatrix{L^{cr}:   \Hom^{cr}(\Gamma, \dG(E))/\dG(E)   \ar[r]^-{\sim} & \Max(R)\ar[r]^-{\sim} & \Max(R^{\univ})
}
\end{equation*}
\end{prop}
\begin{proof} Choose generators $\g_{1},\cdots, \g_{N}$ for $\Gamma$. Then $\Hom(\Gamma,\dG_{E})$ is a closed subscheme of $(\dG_{E})^{N}$ by evaluating at $\g_{1},\cdots, \g_{N}$.  A homomorphism $\rho\in \Hom(\Gamma, \dG(E))$ is completely reducible if and only if the Zariski closure of the group generated by $\rho(\g_{1}),\cdots, \rho(\g_{N})$ is a completely reducible subgroup of $\dG_{E}$. By \cite[Corollary 3.7]{BMR} (which is a combination of \cite[Theorem 16.4]{Rich} and \cite[Theorem 3.1]{BMR}), the latter condition is equivalent to that the $\dG_{E}$-orbit of $(\rho(\g_{1}),\cdots, \rho(\g_{N}))\in (\dG_{E})^{N}$ is closed in $(\dG_{E})^{N}$, which is also equivalent to that the $\dG_{E}$-orbit of $\rho$ is closed in $\Hom(\Gamma,\dG_{E})$. Therefore, $\rho\in \Hom^{cr}(\Gamma, \dG(E))$ if and only if its $\dG_{E}$-orbit is closed in $\Hom(\Gamma,\dG_{E})$. Then the first bijection follows from general properties of the affine case of the GIT quotient \cite[1.3.2]{Rich}.

Now the bijectivity of $L^{cr}$. The maximal ideals of $\Max(R^{\univ})$ correspond to objects $(E, \{\th_{n}\})\in \frR$ where the underlying algebra is $E$ itself. When $\textup{char}(E)=0$, the result of V.Lafforgue in \cite[Proposition 11.7]{Laf} (based on Richardson \cite[Theorem 3.6]{Rich}) says that $(E, \{\th_{n}\})\in \frR$ are in bijection with $\dG(E)$-conjugacy classes of semisimple (same as complete reducible in characteristic zero) representations $\Gamma\to \dG(E)$, hence $L^{cr}$ is a bijection. When $\textup{char}(E)>0$, the same statement is true, see \cite[paragraphs before Th\'eor\`eme 13.2]{Laf}, using again \cite[Corollary 3.7]{BMR}.
\end{proof}

\subsection{Betti Langlands parameters}\label{s:Betti par}

Define the full subcategory 
 \begin{equation*}
 \Sh_{\calN(\LL_S),!}(\Bun(\LL_S), E) \subset
 \Sh_{\calN(\LL_S)} (\Bun(\LL_S), E)
 \end{equation*}
of objects of the form $j_{!}\cF_{\cU}$, where $j:\cU\incl \Bun(\LL_S)$ is an open embedding of a finite type substack, and $\cF_{\cU}$ is a constructible complex on $\cU$, including the traditional requirement that its cohomology sheaves are bounded and finite-rank.

The construction of Betti Langlands parameters now easily follows from the existence of the map $\sigma: R^{\univ}\to \cZ(\Sh_{\calN(\LL_S)} (\Bun(\LL_S), E))$ and Proposition \ref{p:ss rep}.

\begin{theorem-construction}[Betti Langlands parameter]\label{cor: Langlands par} 
Let $E$ be an algebraically closed field. 
To any indecomposable object $\cF\in  \Sh_{\calN(\LL_S),!}(\Bun(\LL_S), E)$, one can canonically attach a $\dG(E)$-local system  $\rho_{\cF}\in \Loc_{\dG}(U)(E)$ whose image has reductive Zariski closure.  

We call $\rho_{\cF}$ the {\em Betti Langlands parameter} of $\cF$.
\end{theorem-construction}
\begin{proof}
Since $\cF$ is constructible with finite type support and indecomposable, the (underived) center  $Z(\End(\cF))$ is a finite-dimensional $E$-algebra without nontrivial idempotents. The image of the map 
$$
\xymatrix{
R^{\univ} \ar[r] & \cZ(\Sh_{\calN(\LL_S)} (\Bun(\LL_S), E))\ar[r] &  Z(\End(\cF))
}
$$ is a local artinian $E$-algebra, which then corresponds to a unique maximal ideal $\fm_{\cF}$ of $R^{\univ}$. We then define $\rho_{\cF}$ to be $L^{cr,-1}(\fm_{\cF})$.
\end{proof}

\subsubsection{Hecke eigensheaves}
Recall that an object $\cF\in \Sh_{\calN(\LL_S)}(\Bun(\LL_S), E)$ is a Hecke eigensheaf with eigenvalue $\rho\in \Loc_{\dG}(U)(E)$, if for any $n\ge1$ and $V\in \Rep((\dG_{E})^{n})$, there is an isomorphism on $\Bun(\LL_S)\times U^{n}$
\begin{equation*}
\xymatrix{  \a_{n,V}:     H_{n, V}(\cF) \simeq   \cF\boxtimes \rho_{V}.
}
\end{equation*}
Here $\rho_{V}$ is the local system on $U^{n}$ given by the representation
\begin{equation*}
\xymatrix{       \pi_{1}(U^{n}, u_{0})=\Gamma^{n}\ar[r]^-{\rho^{n}} & \dG(E)^{n}\ar[r] & \GL(V).
}
\end{equation*}
These isomorphisms $\{\a_{n,V}\}$ are required to satisfy compatibility conditions with the tensor structure on $\Rep((\dG_{E})^{n})$ and factorization structures when passing to diagonals.   For a detailed account of these compatibilities, see \cite[2.8]{GdJ}.

For a homomorphism $\rho:\Gamma\to\dG(E)$, the {\em semisimplification} of $\rho$ is the defined in \cite[3.2.4]{Serre}: choose a parabolic $P\subset \dG_{E}$ that minimally contains $\rho(\Gamma)$, and choose a Levi subgroup $L$ of $P$. Let $\pi:P\to L$ be the projection. The semisimplification $\rho^{ss}$ of $\rho$ is defined, up to $\dG(E)$-conjugation, as the composition
\begin{equation*}
\xymatrix{       \Gamma\ar[r]^{\rho} & P(E)\ar[r]^{\pi} & L(E)\ar@{^{(}->}[r]  & G(E).  \\
}
\end{equation*}

\begin{prop}
Let $E$ be an algebraically closed field. If $\cF\in \Sh_{\calN(\LL_S),!}(\Bun(\LL_S), E)$ is a Hecke eigensheaf with eigenvalue  $\rho\in \Loc_{\dG}(U)(E)$, then the Betti Langlands parameter $\rho_{\cF}$ constructed in Theorem \ref{cor: Langlands par} is isomorphic to the semisimplification of $\rho$.
\end{prop}
\begin{proof}
If $\cF$ is a Hecke eigensheaf with eigenvalue $\rho$, and $f\in \cO((\dG_{E})^{n}/\dG_{E})$, $\un\g\in \Gamma^{n}$, the excursion operator $S_{\un\g,\cF}(f)$ (see \eqref{eq:Ex F}) is the composition
\begin{equation*}
\xymatrix{   \cF \ar[r]^-{c^{\sharp}}\ar[dr]_-{\id\otimes f} & H_{n+1,\cO_{n}}(\cF)_{\D(u_{0})}\ar[r]^{(1,\un\g)} \ar[d]^{\a_{n+1,\cO_{n}}} & H_{n+1, \cO_{n}}(\cF)_{\D(u_{0})} \ar[d]_{\a_{n+1,\cO_{n}}}\ar[r]^-{c^{\flat}} & \cF     \\
& \cF\otimes \cO_{n}\ar[r]^-{R_{\un\g}} & \cF\otimes \cO_{n}\ar[ur]_{\id\otimes\ev_{1}}
}
\end{equation*}
Computing the composition using the lower row (where $R_{\un\g}$ means the right translation action of $\un\g$ on $\cO_{n}=\cO((\dG_{E})^{n})$), we see that $S_{\un\g,\cF}(f)$ acts on $\cF$ by the scalar $f(\rho(\g_{1}),\cdots, \rho(\g_{n}))$.

Let $\rho^{ss}\in \Hom^{cr}(\Gamma, \dG(E))/\dG(E)$ be the semisimplification of $\rho$. Then
\begin{equation*}
f(\rho(\g_{1}),\cdots, \rho(\g_{n}))=f(\rho^{ss}(\g_{1}),\cdots, \rho^{ss}(\g_{n})).
\end{equation*}
On the other hand, by the construction of $\rho_{\cF}$, the image of $S_{\un\g,\cF}(f)\in Z(\End(\cF))$ in the residue field $E$ is $f(\rho_{\cF}(\g_{1}),\cdots, \rho_{\cF}(\g_{n}))$. Therefore
\begin{equation*}
f(\rho_{\cF}(\g_{1}),\cdots, \rho_{\cF}(\g_{n}))=f(\rho^{ss}(\g_{1}),\cdots, \rho^{ss}(\g_{n})), \quad\forall f\in \cO((\dG_{E})^{n}/\dG_{E}),\un\g\in \Gamma^{n}.
\end{equation*}
This implies $L^{cr}(\rho^{ss})=L^{cr}(\rho_{\cF})\in \Max(R^{\univ})$. By Proposition \ref{p:ss rep}, we have $\rho_{\cF}\simeq \rho^{ss}$. 
\end{proof}

%
%\begin{remark}
%In the case of Iwahori level structure $\LL_{S}=(\II_{x})_{x\in S}$, by Theorem~\ref{thm: action tame},
%the map \eqref{eq: end of id plain} in fact lifts to a map
%\begin{equation*}
%\xymatrix{
%\upH^{0}(\Loc_{\dG}(X, S)_{E}, \calO) \ar[r] & \End(\cF).
%}
%\end{equation*} 
%
%\end{remark}

%%%%%%%%%%%%%%%%%%%%%%%%%%%%%%%%%%%%%%%%%%%%%%%%%%%%
%%%%%%%%%%%%%%%%%%%%%%%%%%%%%%%%%%%%%%%%%%%%%%%%%%%%
%%%%%%%%%%%%%%%%%%%%%%%%%%%%%%%%%%%%%%%%%%%%%%%%%%%%

%%%%%%%%%%%%%%%%%%%%%%%%%%%%%%%%%%%%%%%%%%%%%%%%%%%%
%%%%%%%%%%%%%%%%%%%%%%%%%%%%%%%%%%%%%%%%%%%%%%%%%%%%
%%%%%%%%%%%%%%%%%%%%%%%%%%%%%%%%%%%%%%%%%%%%%%%%%%%%


\begin{thebibliography}{99}



\bibitem{AG} D. Arinkin and D. Gaitsgory, 
Singular support of coherent sheaves, and the geometric Langlands conjecture, 
Selecta Math. (N.S.) 21 (2015), no. 1, 1--199. 


\bibitem{BD}
A. Beilinson and V. Drinfeld,
Quantization of Hitchin's integrable system and Hecke eigensheaves, 
available
at {\tt math.uchicago.edu/\~{}mitya}.


\bibitem{BFN} D. Ben-Zvi, J. Francis, and D. Nadler, 
Integral transforms and Drinfeld centers in derived algebraic geometry,
J. Amer. Math. Soc. 23 (2010), 909--966.



\bibitem{BMR}
M. Bate, B. Martin and G. R\"ohrle. 
A geometric approach to complete reducibility.
Invent. Math. 161 (2005), no.1, 177--218 .



\bibitem{BNbetti} D. Ben-Zvi and D. Nadler, 
Betti Geometric Langlands,
arXiv:1606.08523.
  
  


\bibitem{B} R. Bezrukavnikov, 
On two geometric realizations of an affine Hecke algebra,
Publications math\'ematiques de l'IH\'ES,
June 2016, Volume 123, Issue 1, pp 1--67.


\bibitem{Del-tensor}
P.Deligne,
Cat\'egories tannakiennes,
in {\em Grothendieck Festschrift} vol II. Progress in Math. 87,
Birkh\"auser Boston (1990) pp. 111--195.

\bibitem{Gvanish}
D. Gaitsgory,
A generalized vanishing conjecture, 
available at
{\tt http://www.math.harvard.edu/\~{}gaitsgde/GL/GenVan.pdf}.



\bibitem{G}
D. Gaitsgory,
Construction of central elements in the affine Hecke algebra via nearby cycles,
Invent. Math. 144 (2001), no. 2, 253--280.


\bibitem{GdJ}
D. Gaitsgory,
On de Jong's conjecture,
Israel Journal of Math. 157 (2007), 155--191


\bibitem{Gi}
V. Ginzburg,
Perverse sheaves on a Loop group and Langlands duality,
arXiv:math/9511007.

\bibitem{Gin}
V. Ginzburg,
The global nilpotent variety is Lagrangian,
Duke Math. J. 109 (2001), no. 3, 511--519. 

\bibitem{KS}
M. Kashiwara and P. Schapira,
Sheaves on Manifolds,
Grundlehren der Mathematischen Wissenschaften, 292. Springer-Verlag, Berlin, 1990. x+512 pp.


\bibitem{Ko}
J. Koll\'ar,
Lectures on Resolution of Singularities,
 Annals of Math. Studies (166), 2007.



\bibitem{Laf}
V. Lafforgue,
Chtoucas pour les groupes r\'eductifs et param\'etrisation de Langlands globale,
arXiv:1209.5352.

\bibitem{L0}
G. Laumon, 
Correspondance de Langlands g\'eom\'etrique pour les corps de fonctions,
Duke Math. J. 54 (1987), no. 2, 309--359.

\bibitem{L}
G. Laumon,
Un analogue global du c\^one nilpotent,
Duke Math. J. 57 (1988), no. 2, 647--671. 

%\bibitem{LZ1}
%Y. Liu, W. Zheng,
%Enhanced six operations and base change theorem for Artin stacks,
%arXiv:1211.5948.
%
%\bibitem{LZ2}
%Y. Liu, W. Zheng,
%Enhanced adic formalism and perverse t-structures for higher Artin stacks,
%arXiv:1404.1128.

\bibitem{HA}
J. Lurie,
Higher Algebra. Available from {\tt http://www.math.harvard.edu/\~{}lurie/papers/HA.pdf}

\bibitem{MV}
I. Mirkovi\'c and  K. Vilonen, 
Geometric Langlands duality and representations of algebraic groups over commutative rings,
Ann. of Math. (2)  166  (2007),  no. 1, 95--143.




\bibitem{NY3pts}
D. Nadler and Z. Yun,
Geometric Langlands for $\SL(2)$, $\PGL(2)$ over the pair of pants,
arXiv: 1610.08398.

\bibitem{Rich}
R.W. Richardson,
Conjugacy classes of $n$-tuples in Lie algebras and algebraic groups, 
Duke Math. J. 57 (1988), p. 1--35.


\bibitem{RS}
M. Robalo and P. Schapira, 
A lemma for microlocal sheaf theory in the $\infty$-categorical setting, arXiv:1611.06789.


\bibitem{Serre}
J-P.Serre,
Compl\`ete R\'educibilit\'e,
S\'eminaire Bourbaki,  2003-2004, no. 932, p. 195--217.



\bibitem{Sp}
N. Spaltenstein,
Resolutions of unbounded complexes, 
Composito Math. 65 (1988), no. 2, 121--154.

\end{thebibliography}
\end{document}